\documentclass[12pt,reqno]{amsart}
\usepackage[margin=1in]{geometry}
\usepackage{amsmath,amssymb,amsthm,graphicx,amsxtra, setspace}
\usepackage[utf8]{inputenc} 
\usepackage{mathrsfs}
\usepackage{hyperref,eucal}
\usepackage{mathtools}
\usepackage{xcolor}
\usepackage{dsfont}
\usepackage{orcidlink}

\allowdisplaybreaks

\usepackage[normalem]{ulem} 
\newcommand\redsout{\bgroup\markoverwith{\textcolor{red}{\rule[0.5ex]{2pt}{1pt}}}\ULon}

\usepackage{cancel}

\usepackage[pagewise]{lineno}

\usepackage{enumitem}
\setlist[enumerate]{leftmargin=*,widest=0}
\setlist[description]{leftmargin=*,widest=0}

\usepackage{graphicx}
\usepackage{hyperref}
\usepackage{mathtools}

\usepackage[cyr]{aeguill}

\numberwithin{equation}{section}

\colorlet{darkblue}{blue!50!black}

\hypersetup{
	colorlinks,%
	citecolor=blue,%
	filecolor=red,%
	linkcolor=darkblue,%
	urlcolor=blue,%
	pdfnewwindow=true,%
	pdfstartview={FitH}
}

\colorlet{darkblue}{red!100!black}

\newtheorem{theorem}{Theorem}[section]
\newtheorem{lemma}[theorem]{Lemma}

\newtheorem{definition}[theorem]{Definition}

\newtheorem{remark}[theorem]{Remark}

\newtheorem{claim}{Claim}

\newcommand{\Tr}{\mathop{\mathrm{Tr}}}

\let\originalleft\left
\let\originalright\right
\renewcommand{\left}{\mathopen{}\mathclose\bgroup\originalleft}
\renewcommand{\right}{\aftergroup\egroup\originalright}

\theoremstyle{definition}

\def\1{\Omega}

\def\d{\mathrm{d}}

\def\dx{\,\mathrm{d}\boldsymbol{x}}
\def\dt{\,\mathrm{d}t}
\def\ds{\,\mathrm{d}s}
\def\dxt{\,\mathrm{d}\boldsymbol{x}\mathrm{d}t}

\def\div{\operatorname{div}}
\def\divv{\mathbf{div}\,}
\def\supp{\operatorname{supp}}
\def\I{\mathrm{I}}

\def\D{\mathrm{D}}

\def\W{\mathrm{W}}
\def\R{\mathbb{R}}
\def\E{\mathrm{E}}

\def\loc{\mathrm{loc}}

\def\L{\mathrm{L}}
\def\HH{\mathrm{H}}
\def\uu{\boldsymbol{u}}  
\def\vv{\boldsymbol{v}}
\def\w{\boldsymbol{w}}

\def\f{\boldsymbol{f}}
\def\00{\boldsymbol{0}}
\def\C{\mathrm{C}}

\def\N{\mathbb{N}}
\def\x{\boldsymbol{x}}

\def\bfb{\mathbf{b}}
\def\bfB{\mathbf{B}}
\def\bfC{\mathbf{C}}
\def\bfX{\mathbf{X}}

\def\bfF{\mathbf{F}}

\def\bfH{\mathbf{H}}
\def\bfL{\mathbf{L}}

\def\bfQ{\mathbf{Q}}

\def\bfV{\mathbf{V}}
\def\bfI{\mathbf{I}}
\def\bfD{\mathbf{D}}
\def\bfQ{\mathbf{Q}}
\def\bfE{\mathbf{E}}

\def\bfA{\mathbf{A}}
\def\bfE{\mathbf{E}}
\def\bfX{\mathbf{X}}
\def\bfb{\mathbf{b}}
\def\bfc{\mathbf{c}}

\def\bf0{\mathbf{0}}

\def\bmcA{\boldsymbol{\mathcal{A}}}

\def\rw{\rm{w}}

\usepackage{bm}

\def\bfi{\boldsymbol{\varphi}}

\def\bpsi{\boldsymbol{\psi}}
\def\bxi{\boldsymbol{\xi}}

\def\bVcal{\boldsymbol{\mathcal{V}}}

\DeclareMathAlphabet\mathbfcal{OMS}{cmsy}{b}{n}
\def\bcG{\mathbfcal{G}}
\def\bcH{\mathbfcal{H}}
\def\bcF{\mathbfcal{F}}

\def\Tr{\operatorname{Tr}}

\usepackage{comment}
\usepackage{amsmath}
\usepackage{autobreak}

\usepackage{amsrefs}

\newcommand{\Addresses}{{
		\footnote{
			\noindent \textsuperscript{1}Montanuniversit\"at Leoben, Department Mathematik und Informationstechnologie, Franz Josef Strasse 18, 8700
			Leoben, Austria.\\
			\textsuperscript{2}{FCT - Universidade do Algarve, Faro, Portugal \& CIDMA - Universidade de Aveiro, Aveiro, Portugal.}\\
			\textsuperscript{3}Department of Mathematics, Indian Institute of Technology Roorkee-IIT Roorkee,
			Haridwar Highway, Roorkee, Uttarakhand 247667, INDIA.			
			\par\nopagebreak
			\textit{e-mail:} \texttt{Ankit Kumar: ankitkumar.2608@gmail.com, ankit.kumar@unileoben.ac.at.}\\
				\textit{e-mail:} \texttt{Hermenegildo Borges de Oliveira: holivei@ualg.pt.}\\
			\noindent  \textit{e-mail:} \texttt{Manil T. Mohan: maniltmohan@ma.iitr.ac.in, maniltmohan@gmail.com.}
			
				$^\ast$Corresponding author.

			\textit{Keywords:} Generalised Navier-Stokes-Voigt equations, fluid dynamics, non-Newtonian fluids, weak solution.

			Mathematics Subject Classification (2020): {Primary 35Q35, 76A05, 76D03; Secondary 76D05, 35D30, 74F10.}

}}}
\begin{document}	
	
\title[Optimal existence of weak solutions for the power-law NSV equations]{Optimal existence of weak solutions for the power-law Navier-Stokes-Voigt equations
		\Addresses}

	\author[A.~Kumar, H.~B.~de Oliveira and M.~T.~Mohan]
	{Ankit Kumar\orcidlink{0000-0002-8791-0765}\textsuperscript{1}, Hermenegildo Borges de Oliveira\orcidlink{0000-0001-9053-8442}\textsuperscript{2,$\ast$} and Manil T. Mohan\orcidlink{0000-0003-3197-1136}\textsuperscript{3}}

	\maketitle

%
%

    \date{\today}	

	\begin{abstract}
We study the power-law Navier--Stokes--Voigt system in a bounded domain, subject to appropriate initial conditions and no-slip boundary conditions.
The analysis is governed by the power-law exponent $p$, ranging from shear-thinning ($1<p<2$) through the Newtonian case ($p=2$) to shear-thickening ($p>2$).
We prove the global existence of weak solutions to the associated initial-boundary value problem. These results are optimal in the sense that they hold for the full admissible range of the exponent $p >1$ and for any space dimension $d\geq 2$.
For $p>1$ when $2 \leq d \leq 3$, and for $p \geq \tfrac{d}{2}$ when $d>3$, existence follows from a classical approach combining compactness arguments with the theory of monotone operators.
The remaining case, $1<p<\tfrac{d}{2}$ with $d>3$, requires in addition a pressure decomposition technique and an $\L^\infty$--truncation argument. {Finally, we establish the uniqueness of weak solution for $p>1$ when $2\leq d\leq 4$, and for  $p\geq \frac{d}{2}$ when $d>4$.}

	\end{abstract}
	
	\maketitle

\section{Introduction}\label{sec1}

	\subsection{The model}
{   We consider the power-law Navier–Stokes–Voigt equations in a bounded domain $\Omega \subset \mathbb{R}^d$, $d \ge 2$, on the time interval $[0, T]$.
This model regularizes the standard power-law Navier–Stokes equations through the addition of the relaxation (Kelvin-Voigt) term $-\kappa\Delta\partial_t \vv$ in the momentum equation.
The initial-boundary value problem under consideration is given by the system:}
\begin{equation}\label{1.1}
	\left\{	
	\begin{aligned}
	{\partial_t(\vv -\kappa \Delta \vv ) +\divv(\vv\otimes \vv)+\nabla \pi -\nu\,\divv\!\left(|\bfD(\vv)|^{p-2}\bfD(\vv)\right)} & = {\f },\quad \text{ in }Q_T,  \\
		\div\vv &=0,\quad \text{ in }Q_T,  \\
		\vv &=\vv_0,\quad \text{ in }\1\times\{0\},  \\
		\vv &=\boldsymbol{0},\quad \text{ on }\Gamma_T,
	\end{aligned}
\right.
\end{equation}
where $Q_T:=\1\times(0,T)$ and $\Gamma_T:=\partial\1\times(0,T)$, and $\partial\1$ denotes the boundary of $\Omega$.
Here, $\vv=(v_1,\dots,v_d)$ and $\f=(f_1,\dots,f_d)$ are vector-valued functions, $\pi $ is a scalar-valued function, whereas $\nu$ and $\kappa$  are positive constants.
In real-world applications, space dimensions of interest are $d=2$ and $d=3$, and in that case $\vv$ accounts for the velocity field, $\pi $ is the pressure (divided the constant density),
$\f$ represents an external forces field.
Greek letters $\nu$ and $\kappa$ stay for positive constants that, in the dimensions of physics interest, account for the kinematics viscosity and for the relaxation time, that is
 the time required for a viscoelastic fluid relax from a deformed state back to its equilibrium configuration.
The power-law exponent $p$ is a constant that characterizes the flow and is assumed to be such that $1<p<\infty$.
	Lowercase boldface letters denote vector quantities, whereas uppercase boldface letters denote tensor-valued quantities.
In particular, $\bfI$ is the unit tensor and $\bfD(\vv)$ is the symmetric part of the velocity gradient, i.e.~$\bfD(\vv):=\frac{1}{2}\left(\nabla \vv + (\nabla \vv)^{\top}\right)$.
The notation ${\partial_t\vv}$ stays, as usual, for the partial time derivative of $\vv$.

Equation \eqref{1.1}$_1$ can be derived from the principle of conservation of momentum by considering the deviatoric part of the Cauchy stress tensor of the form
\begin{equation}\label{KV:law}
\boldsymbol{\tau}
  = 2\nu\,|\bfD(\vv)|^{p-2}\,\bfD(\vv)
    + 2\kappa\,\partial_t\bfD(\vv),
\end{equation}
which represents a nonlinear viscous response with an additional linear elastic contribution.
When $\kappa=0$, this extra stress reduces to the pure power-law ({Ostwald--de Waele}) relation
\begin{equation}\label{dev:str:OdW}
\boldsymbol{\tau}
  = 2\nu\bigl(|\bfD(\vv)|\bigr)\,\bfD(\vv),
\qquad
\nu(s) = \nu_0\, s^{p-2}
\quad (p>1,\ \nu_0>0),
\end{equation}
so that the apparent viscosity scales like $|\bfD(\vv)|^{p-2}$.
In this setting, shear-thinning (or pseudo-plastic) fluids correspond to exponents $1<p<2$, shear-thickening (or dilatant) fluids to $p>2$, whereas the special case $p=2$ recovers the classical Navier-Stokes equations for Newtonian fluids.
When $p=2$ and $\kappa>0$ (with $\nu>0$) in the constitutive law \eqref{KV:law},
the deviatoric stress tensor reflects a combined effect of
Newton's law of viscosity and Hooke's law of linear elasticity.
For this case, equation~\eqref{KV:law} generalizes the one-dimensional Kelvin-Voigt model to higher dimensions and has been successfully applied to describe the rheology of viscous fluids with elastic properties, such as polymer solutions and certain biological tissues~\cite{Pav-1971,CSS:2002}.
The classical one-dimensional Kelvin-Voigt model,
\begin{equation}\label{eq.1.20}
\tau
  = \nu\,\frac{\partial v}{\partial x}
    + \kappa\,\frac{\partial^2 v}{\partial x \,\partial t},
\end{equation}
provides a simple constitutive relation for linear viscoelastic materials.
This model can be obtained either from the Boltzmann superposition principle or through mechanical analogs consisting of elastic springs and viscous dashpots~\cite{Barnes:2000}.
A natural power-law generalization of \eqref{eq.1.20}, accounting for non-Newtonian viscous fluids endowed with (linear) elastic properties, is given by
\begin{equation}\label{Plaw:1d:KV}
\tau
  = \nu\,\left| \frac{\partial v}{\partial x} \right|^{p-2}\,\frac{\partial v}{\partial x}
    + \kappa\,\frac{\partial^2 v}{\partial x \,\partial t}.
\end{equation}

In fluid mechanics, the choice of a constitutive equation
is largely directed  by the specific application under consideration.
For describing the rheological behavior of many incompressible viscoelastic
fluids, equations of the type \eqref{eq.1.20} and \eqref{Plaw:1d:KV}
are widely regarded as a natural starting point~\cite{BAH:1987a}.
Since a great variety of stress--strain relations may be postulated,
it is customary to impose admissibility criteria in order to restrict
attention to physically meaningful models. For incompressible fluids,
{Oldroyd}~\cite{Oldroyd:1950} proposed that constitutive laws should be:
(a) frame-invariant, (b) deterministic, and (c) local. However,
the elastic tensor $\partial_t\bfD(\vv)$ fails to be objective and thus
does not satisfy the principle of frame indifference.
To illustrate this, let $\bfD(\vv)^\ast$ and
$\partial_t\bfD(\vv)^\ast$ denote the representations of
$\bfD(\vv)$ and $\partial_t\bfD(\vv)$ in another reference frame,
and let $\bfQ=\bfQ(t)$ be an arbitrary orthogonal tensor. Since
$\bfD(\vv)$ is objective, we have
$\bfD(\vv)^\ast=\bfQ\bfD(\vv)\bfQ^\top$.
Objectivity of $\partial_t\bfD(\vv)$ would then require
\begin{equation*}
\partial_t\bfD(\vv)^\ast
=\partial_t\bfQ\bfD(\vv)\bfQ^\top
+\bfQ\partial_t\bfD(\vv)\bfQ^\top
+\bfQ\bfD(\vv)\partial_t\bfQ^\top,
\end{equation*}
so that
\begin{equation*}
\partial_t\bfQ\bfD(\vv)\bfQ^\top+\bfQ\bfD(\vv)\partial_t\bfQ^\top=\00,
\end{equation*}
a condition not generally satisfied. To obtain an admissible model,
the partial time derivative in \eqref{dev:str:OdW} can be replaced by
the convected time derivative~\cite{BAH:1987a}
\begin{equation}\label{eq.1.22}
\bfD(\vv)^\triangledown
=\partial_t\bfD(\vv)
+(\vv\cdot\nabla)\bfD(\vv)
-\Big(\bfD(\vv)\nabla\vv+(\bfD(\vv)\nabla\vv)^\top\Big).
\end{equation}
However, the PDEs obtained from constitutive relations involving
\eqref{eq.1.22} are highly nonlinear and analytically difficult,
so certain simplifications are required to make a rigorous
mathematical study feasible.

\subsection{Literature review}
The mathematical analysis of pure ($p\not=2$) power-law fluids was initiated by {Ladyzhenskaya}~\cite{Lady-Steklov-1967,Lady-LOMI-1968}, followed by {Lions}~\cite{Lions:1969}.
In her celebrated papers~\cite{Lady-Steklov-1967,Lady-LOMI-1968}, {Ladyzhenskaya} proposed studying what she termed the modified Navier--Stokes equations, in which the deviatoric stress tensor is prescribed by~\eqref{dev:str:OdW}. Models of this kind had already been investigated experimentally since the early twentieth century in connection with polymer melts, colloidal suspensions, and other non-Newtonian fluids (see e.g.~\cite{Barnes:2000}). By adopting this form of the stress--strain relation, {Ladyzhenskaya} naturally placed her mathematical framework in line with experimentally observed non-Newtonian behavior, even though her main goal was to establish solvability results rather than to develop a detailed physical theory.
In~\cite{Lady-Steklov-1967,Lady-LOMI-1968}, it was proved
the existence of 3d weak solutions for
\begin{equation*}
p\geq \frac{11}{5},
\end{equation*}
and their uniqueness for
\begin{equation*}
p\geq \frac{5}{2},
\end{equation*}
Subsequently, {Lions}~\cite{Lions:1969} extended these results to the case of a general space dimension $d\geq 2$, and where the diffusion depends on the full velocity gradient $\nabla\vv$ in place of its symmetric part $\bfD(\vv)$. He established the existence for
\begin{equation}\label{Lady:Li:p}
p\geq\frac{3d+2}{d+2},
\end{equation}
and uniqueness for
\begin{equation*}
p\geq \frac{d+2}{2}.
\end{equation*}
The existence results in~\cite{Lady-Steklov-1967,Lady-LOMI-1968,Lions:1969}, in addition to employing Galerkin approximations combined with compactness arguments and the theory of monotone operators, rely (directly or indirectly) on the continuous embedding

\begin{equation}\label{imb-LL}
\L^p(0,T;\mathbf{V}_p)\cap\L^{\infty}(0,T;\mathbf{H})\hookrightarrow\L^{\frac{p(d+2)}{d}}(0,T;\L^{\frac{p(d+2)}{d}}(\Omega)^d),
\end{equation}
which ensures that $\vv\otimes\vv:\bfD(\bfi)$ is bounded in $\bfL^{1}(Q_T)$
for all $\vv$ and $\bfi$ in the same function spaces whenever \eqref{Lady:Li:p} holds.
See the next section for the definition of the function spaces appearing in~\eqref{imb-LL}.

More than forty years later, {Zhikov}~\cite{Zhikov-2009} succeeded in improving the existence results~\cite{Lady-Steklov-1967,Lady-LOMI-1968,Lions:1969} to the case of
\begin{equation*}
p\geq \max\left\{\frac{3d}{d+2},\frac{d+\sqrt{3d^2+4d}}{d+2}\right\}.
\end{equation*}
The main idea of his proof consisted in a regularization of the deviatoric stress tensor followed by the application of the classical existence results~\cite{Lady-Steklov-1967,Lady-LOMI-1968,Lions:1969}.
The passage to the limit in the regularized problem is justified using arguments from Measure Theory, which allow one to handle the weak convergence of nonlinear terms.
It should be stressed that this lower bound for $p$ had already been obtained by {Ne\v{c}as and his collaborators}~\cite{Necas:book:1996}, although their analysis was carried out in the setting of periodic boundary conditions instead.

Even prior to Zhikov’s result, {Wolf}~\cite{Wolf:2007} established the existence of weak solutions to the Navier-Stokes system with a power-law viscosity for exponents
\begin{equation*}
p > \frac{2(d+1)}{d+2}.
\end{equation*}
In his proof, {Wolf} employed a decomposition of the pressure into a regular (measurable) part and a singular part, which allowed him to handle the nonlinear convective term effectively. This was combined with the $\L^\infty$--truncation method, a technique that enables control of the nonlinearities and provides compactness necessary to pass to the limit in the weak formulation. We exploit the same idea to establish our results.

Finally, {Diening et al.}~\cite{DRW} extended and improved the existence result~\cite{Wolf:2007} to cover smaller values of the power-law exponent $p$, specifically
\begin{equation*}
\frac{2d}{d+2} < p < 2.
\end{equation*}
In this work, the main tool was the successful combination of the pressure decomposition technique~\cite{Wolf:2007} with the Lipschitz--truncation method, which provides strong control over approximate solutions. This approach enabled the authors to extend the existence theory of weak solutions to a wider range of power-law exponents, particularly in the shear-thinning case.

Concerning existence, the range
\begin{equation*}
1 < p < \frac{2d}{d+2},\quad \mbox{for}\ \ d>3,
\end{equation*}
remains out of reach with the techniques currently available.
The main difficulty lies in the lack of an appropriate Gelfand triple that would allow the application of the Aubin-Dubinski\u{\i} compactness lemma to extract strong convergence of approximate solutions.

With regard to uniqueness for the Navier-Stokes system with a power-law viscosity, the classical results of~\cite{Lady-Steklov-1967,Lady-LOMI-1968,Lions:1969} have only been improved (for $d\geq 3$) up to the range
\begin{equation*}
p \geq \frac{3d+2}{\,d+2\,}.
\end{equation*}
In particular, see the work by {Du and Gunzburger}~\cite{DuGunzburger} as well as the contributions by {Ne\v{c}as and his collaborators}~\cite{Necas:book:1996}, where refined energy methods were used to establish uniqueness of weak (or suitably regular) solutions.

The particular case of the momentum \eqref{1.1}$_1$ with $p=2$ is usually referred to as the Navier-Stokes-Voigt (NSV) equations.
Over the past several decades, the corresponding problem and related models have attracted considerable attention in the mathematical literature due to their analytical tractability and relevance to applications.
The first rigorous study of existence and uniqueness of weak solutions for the NSV equations is due to {Oskolkov}~\cite{Oskolkov:1973}, who coined the name Kelvin-Voigt (initially spelled Kelvin-Voight) for this type of regularized fluid model.
As noted by {Zvyagin and Turbin}~\cite{zvy-2010}, neither {Kelvin} nor {Voigt} themselves proposed a constitutive stress--strain law or a complete set of governing equations for viscoelastic fluids.
Nevertheless, the name Navier-Stokes-Voigt has become widely accepted, since this model naturally extends the equations originally suggested by {Voigt} (for elastic media with relaxation effects) to encompass viscous fluids with viscoelastic characteristics.
From the viewpoint of applications, one of the earliest uses of this model appears to be due to {Pavlovsky}~\cite{Pav-1971}, who employed a simplified form of the equations to describe weakly concentrated water--polymer mixtures.
More recently, {Ebrahimi et al.}~\cite{EHL:2013} explored the two-dimensional NSV equations as a tool in numerical algorithms, including applications to image inpainting, illustrating the flexibility of the model beyond classical fluid dynamics.
Since the  pioneering contributions of {Oskolkov}~\cite{Oskolkov:1973}, the NSV equations have been investigated extensively with respect to existence, uniqueness, regularity, and asymptotic properties of solutions.
A key analytical feature, first emphasized by {Ladyzhenskaya}~\cite{Ladyzhenskaya:1966} and later by {Oskolkov} himself, is that the additional Voigt term $-\kappa\partial_t\bfD(\vv)$ acts as a regularizing mechanism for the Navier-Stokes system: the three-dimensional transient problem possesses a unique global weak solution.
Before the year 2000, most of the research on this model was carried out by the Russian School of Analysis (see the survey~\cite{zvy-2010} and references therein).
In the following years, the NSV equations have been further examined by many authors in various directions.
{Cao et al.}~\cite{CLT:2006} proposed the system as a physically motivated regularization (for small $\kappa$) of the three-dimensional Navier-Stokes equations for purposes of direct numerical simulation, both under periodic and no-slip Dirichlet boundary conditions.
Several works have addressed the long-time dynamics of solutions:
Kalantarov and Titi~\cite{KT:2009}, and {Coti Zelati and Gal}~\cite{C-ZG:2015} established existence and regularity of global attractors,
while {Dou and Qin}~\cite{DYQ:2011} and {Garcia-Luengo et al.}~\cite{G-LM-Rr:2012} analyzed uniform and pullback attractors.
{Berselli and Spirito}~\cite{BS:2017} constructed suitable weak solutions for the 3d Navier-Stokes equations via the Voigt approximation.
The decay rates of solutions in relation to the initial data have been characterized by {Niche}~\cite{Niche:2016} and {Zhao and Zhu}~\cite{ZZ:2015}.
Connections to turbulence modeling have also been highlighted by {Lewandowski et al.}~\cite{Lewandowski:2018}.
{Damázio et al.}~\cite{DMS:2016} establish an $\L^q$--theory for the NSV equations in bounded domains, proving existence, uniqueness, and regularity of solutions, while {Baranovskii}~\cite{Baranovskii:2023} analyzes the NSV equations with position-dependent slip boundary conditions, addressing well-posedness and the impact of such boundaries.
{Mohan}~\cite{Mohan:2020} provided a comprehensive study of global solvability, uniqueness, asymptotic behavior, and certain control problems for the 3d Navier-Stokes-Voigt equations with memory effects.
{Antontsev et al.}~\cite{Ant-Oliv-Khom-2021,Ant-Oliv-2022a,Ant-Oliv-Khom-2022b} studied NSV equations with variable density, proving existence, uniqueness, and regularity, including cases with unbounded density, and extended the analysis to long-time behavior, anisotropic viscosity, relaxation, and damping.
{De Oliveira et al.}~\cite{deOliveira:2025} established the existence and uniqueness of strong solutions for the NSV equations with non-negative (possibly vanishing) density, providing detailed regularity and well-posedness results.
{Zvyagin and Turbin}~\cite{Zvy-Turbin-2023a,Zvy-Turbin-2025} also established weak solvability for NSV equations with non-constant density, including finite-order models, even when the initial density is not strictly positive.

The power-law case, with the exponent $p$ varying over the entire interval $(1,\infty)$, has been investigated by relatively few authors to date.
{Antontsev et al.}~\cite{Ant-Oliv-Khom-2019,Ant-Oliv-Khom-2021b,Ant-Khom-2017} studied the Navier-Stokes-Voigt equations incorporating a $p$--Laplacian or anisotropic diffusion (allowing the power-law exponent to differ along distinct spatial directions), as well as relaxation and damping terms, and provided a detailed analysis of existence, uniqueness, conditions for finite-time blow-up, and the large-time behavior of solutions.
In a stochastic setting, {Kumar et al.}~\cite{Kumar:2025} established existence and uniqueness of weak solutions for the generalised (power-law) stochastic Navier-Stokes-Voigt equations, thereby providing a rigorous mathematical framework for studying this model under random forcing and highlighting the impact of stochasticity on the solution dynamics.

\subsection{Objectives and novelties of the paper}
The main goal of this article is to demonstrate the existence of weak solutions to the generalised Navier--Stokes--Voigt equations. The literature on this model appears to be limited, and to the best of our knowledge, the analysis presented here has not yet been explored.
\begin{itemize}
	\item
As mentioned above, there are few works in the literature addressing power-law fluids \cite{Ant-Oliv-Khom-2019,Ant-Oliv-Khom-2021b,Ant-Khom-2017,DRW, Wolf:2007}. In \cite{Wolf:2007}, the authors studied power-law fluids and established existence results for the exponent $p >\frac{2(d+1)}{d+2}$ using the $\L^\infty$-truncation, while \cite{DRW} extended these results to $p >\frac{2d}{d+2}$ via Lipschitz-truncation.
In our analysis, the presence of the Voigt term $-\kappa \partial_t\Delta\vv$ allows us to establish existence of weak solutions of the generalised Navier--Stokes--Voigt equations for the power-law exponents satisfying $p>1$ for $d\geq 2$, using compactness, monotonicity and $\L^\infty$--truncation arguments. {Moreover, we demonstrate the uniqueness of weak solution to the problem \eqref{1.1}, for $p>1$ when $2\leq d\leq 4 $, and for  $p\geq \frac{d}{2}$ when $d>4$.}
\item Another novelty of this article is that we do not employ the Helmholtz--Hodge projection, which would restrict our analysis to the case $p \geq 2$. Although the Helmholtz--Hodge projection is well-defined for $1<p <2$, the analysis in this range is highly delicate and creates significant difficulties in controlling the convective term in the momentum equation \eqref{1.1}$_1$. {By avoiding the Helmholtz--Hodge projection, we are able to treat the full ranges {$p>1$ for $d\geq 2$,} directly, without encountering these technical obstacles.}
\end{itemize}

\subsection{Organization}
{The structure of this article is organized as follows:
Section \ref{sec2} lays the groundwork by introducing the necessary function spaces and auxiliary results essential for the analysis.   In Section \ref{main:res}, we present the definition of the weak solution (Definition \ref{def.1}) and state our main results, including Theorems \ref{thm:exist:1}, \ref{thm:exist:2} and \ref{UniT}.
  Section \ref{Sect:ex:thm:1} focuses on the proof of Theorem \ref{thm:exist:1} for the cases of $p>1$ when $2\leq d\leq 3$, {and $p\geq\frac{d}{2}$ when $d>3$. }
  This section leverages key tools such as the Banach-Alaoglu theorem (Subsection \ref{WC}) and the Minty trick (Subsection \ref{MT}).
  In Section \ref{PD}, we address the decomposition of the pressure term, ensuring that each decomposed component corresponds to a distinct equation, as formalized in Theorem \ref{thrm41}.
 The proof of Theorem \ref{thm:exist:2} {for $1<p<\frac{d}{2}$ when $d>3$}, is carried out in Section \ref{Sect:ex:small}.
 This section introduces the regularized problem \eqref{1.1:AS1} and its corresponding definition (Definition \ref{def.1:ap}).
 {It also reconstructs the pressure term (Subsection~\ref{Subs:Press:Dec}), and employs compactness arguments (Subsection~\ref{CP}), the $\L^\infty$--truncation technique (Subsection~\ref{Linf:T}), and the monotonicity method (Subsection~\ref{MiT}).} Finally, the article concludes with Section \ref{Sect:unique}, where we establish the uniqueness of the weak solution for $p>1$ when $2 \leq d\leq 4$, and for  $p\geq \frac{d}{2}$ when $d>4$, as stated in Theorem \ref{UniT}.
}

\section{Auxiliary results}\label{sec2}
\subsection{Function spaces}
Let us  introduce the functions spaces used to characterize the solutions of Fluid Mechanics problems.	Let $\1$ be a domain in $\R^d$, that is,  an open and connected subset of $\R^d$.	
By $\C_0^{\infty}(\1)^d$ we  denote the space of all infinitely differentiable $\R^d-$valued functions with compact support in $\1$.	For $1\leq p\leq \infty$, we denote by $\L^p(\1)^d$ the Lebesgue space consisting of all $\R^d-$valued measurable (equivalent classes of) functions that are $p-$summable over $\1$.	
By $|\1|$ we denote the $d$--Lebesgue measure of $\1$, {and $\lambda^{d}$ accounts for the $d$-dimensional Lebesgue measure.}
The corresponding Sobolev spaces are represented by $\W^{k,p}(\1)^d$, where $k\in\N$.	
When $p=2$, $\W^{k,2}(\1)^d$ are Hilbert spaces that we denote by $\HH^k(\1)^d$.	We define
	\begin{align}
		\bVcal &:=\big\{\vv\in\C_0^{\infty}(\1)^d:\nabla\cdot \vv=0\big\}, \label{V-cal} \\
		\bfL_\sigma^p(\1)&:= \text{ the closure of }  \bVcal \text{ in the Lebesgue space } \L^p(\1)^d, \label{L-sig} \\
		\bfV_p^k&:= \text{ the closure of }  \bVcal \text{ in the Sobolev space } \W^{k,p}(\1)^d, \label{V-p:k}
	\end{align}
	for $1\leq p<\infty$ and $k\in\N$.
	In the particular case of $p=2$, we denote the spaces $\bfL_\sigma^p(\1)$ and $\bfV_p^k$ by $\bfH$ and $\bfV^k$, respectively.
If $k=1$, we denote $\bfV_p^k$ by $\bfV_p$, and if both $p=2$ and $k=1$, we denote $\bfV_p^k$ solely by $\bfV$.
	Let $(\cdot,\cdot)$ stand for the inner product of the Hilbert space $\L^2(\1)^d$, and we denote by $\langle \cdot,\cdot\rangle ,$ the induced duality product between $\W_0^{1,p}(\1)^d$  and its dual space $\W^{-1,p'}(\1)^d$, as well as between $\L^p(\1)^d$ and its dual $\L^{p'}(\1)^d$, where $\frac{1}{p}+\frac{1}{p'}=1$.
In the sequel, the $\L^p$, $\W^{k,p}$ and $\W^{-k,p'}$ norms will be denoted in short by $\|\cdot\|_p$, $\|\cdot\|_{k,p}$ and $\|\cdot\|_{-k,p'}$.
On $\bfV$, we consider the equivalent norm $\|\vv\|_{\bfV}:=\|\nabla\vv\|_{2}$, $\vv\in\bfV$.

Given a Banach space $\bfB$, we denote by $\L^r(0,T;\bfB)$, with $1\leq r\leq \infty$, {the} space consisting of all Bochner  $\bfB-$valued measurable (equivalent classes of) functions that are $r-$summable over $[0,T]$.
The norm on these spaces is denoted by $\Vert{}\cdot\Vert{}_{\L^r(0,T;\bfB)}$ for scalar, vector, and tensor functions alike.
In particular, $\C_{\rw}([0,T];\bfB)$ denotes the subspace of $\L^\infty(0,T;\bfB)$ formed by weakly continuous functions from $[0,T]$ into
$\bfB$.
We also consider the following function space
\begin{equation*}
\bfC^\infty_{0,\div}(Q_T):=\left\{\bfi\in \C^{\infty}(Q_T)^d : \div\bfi=0\ \text{ with } \supp(\bfi)\Subset\1\times[0,T)\right\}.
\end{equation*}

\subsection{Auxiliary results}
	We recall the following inequalities which are classical in the theory of $p$-Laplace equations.
	\begin{lemma}\label{GM}
		For all $\bfE,\ \bfF\in\R^{d\times d}$, the following assertions hold true:
		\begin{alignat}{2}
			\label{2.2}
			  \frac{1}{2^{p-1}}|\bfE-\bfF|^p&\leq \big(|\bfE|^{p-2}\bfE-|\bfF|^{p-2}\bfF\big):\big(\bfE-\bfF\big), && \text{ for } p\in [2,\infty),\\
			\label{2.3}
			 (p-1)|\bfE-\bfF|^2&\leq \big(|\bfE|^{p-2}\bfE-|\bfF|^{p-2}\bfF\big):(\bfE-\bfF)\big(|\bfE|^p+|\bfF|^p\big)^{\frac{2-p}{p}},&& \text{ for } p\in(1,2).
 		\end{alignat}
	\end{lemma}
	\begin{proof}
The proof combines \cite[Lemmas 5.1 and 5.2]{GM:1975}.
	\end{proof}
	
	From (\ref{2.2}) and (\ref{2.3}), one easily gets
	\begin{align}\label{2.4}
		\big(|\bfD(\vv)|^{p-2}\bfD(\vv)-|\bfD(\uu)|^{p-2}\bfD(\uu)\big):\big(\bfD(\vv)-\bfD(\uu)\big)\geq 0,\  \text{ for all }\  \vv,\uu\in \bfV_p,
	\end{align}
	and whenever $1<p<\infty$.
	As a consequence of (\ref{2.4}), the operator
\begin{equation}\label{op:A(u)}
\bfA(\vv):=|\bfD(\vv)|^{p-2}\bfD(\vv)
\end{equation}
is said to be {monotone} for any $p$ such that $1<p<\infty$.

The following result generalises {Simon}~\cite[Corollaries~8 and 9]{JS} versions of Aubin-Dubinski\u{\i} compactness lemma.
Right below, we assume $p_0, \ p_1\in[1,\infty]$, $s_0,\ s_1\in\R$ and $T>0$, denote $\partial_t x$ as the distributional time derivative of $x$, and define $p_\theta$ and $s_\theta$ as follows
\begin{equation*}
\frac{1}{p_\theta}:=\frac{1-\theta}{p_0}+\frac{\theta}{p_1},\quad s_\theta:=(1-\theta)s_0+\theta s_1,\quad \theta\in(0,1).
\end{equation*}
\begin{lemma}\label{lem:Sim:Am}
Let $\E$, $\E_0$ and $\E_1$ be Banach spaces such that 
\begin{equation}\label{int:theta}
\E_1\hookrightarrow\hookrightarrow \E\hookrightarrow \E_0\quad \wedge\quad \|x\|_{\E}\leq C\|x\|_{\E_0}^{1-\theta}\|x\|_{\E_1}^\theta, \ x\ \in \E_1.
\end{equation}
\begin{enumerate}
\item Then
\begin{alignat*}{2}
& \W^{s_0,p_0}(0,T;\E_0)\cap \W^{s_1,p_1}(0,T;\E_1)\hookrightarrow\hookrightarrow \W^{s,p}(0,T;\E),
\quad\mbox{if}\quad s<s_\theta\ \wedge\ s-\frac{d}{p}<s_\theta-\frac{d}{p_\theta}; && \\
& \W^{s_0,p_0}(0,T;\E_0)\cap \W^{s_1,p_1}(0,T;\E_1)\hookrightarrow\hookrightarrow \C^s([0,T];\E),\quad\mbox{if}\quad 0\leq s<s_\theta-\frac{d}{p_\theta}.
\end{alignat*}
\item If, in addition to \eqref{int:theta},
\begin{alignat*}{2}
& \mbox{$\mathrm{X}$ is a bounded subset of $\L^{p_1}(0,T;\E_1)$}, && \\
& \mbox{$\partial \mathrm{X}:=\left\{\partial_t x: x\in \mathrm{X}\right\}$ is bounded in $\L^{p_0}(0,T:\E_0)$},
\end{alignat*}
then $\mathrm{X}$ is relatively compact in:
\begin{alignat*}{2}
& \W^{s,p}(0,T;\E),\quad\mbox{if}\quad 0\leq s<1-\theta\ \wedge\ s-\frac{1}{p}<1-\theta-\frac{1}{p_\theta}; && \\
& \C^s([0,T];\E),\quad\mbox{if}\quad 0\leq s<1-\theta-\frac{1}{p_\theta}.
\end{alignat*}
\end{enumerate}
\end{lemma}
\begin{proof}
For the proof of (1), we address the reader to \cite[Theorem~5.2]{Amann:2000}, and for (2) to \cite[Theorem~1.1]{Amann:2000}.
See also the proofs of \cite[Corollaries~8 and 9]{JS}.
\end{proof}

The following result provides a variant of the classical de~Rham theorem.
\begin{lemma}\label{lemm:BP}
	Assume $1<\gamma<\infty$.
 For each $\uu^\ast\in\W^{-1,\gamma}(\1)^d$ such that
		\begin{equation*}
			\big< \uu^{\ast},\uu\big>=0, \quad \text{ for all }\ \uu\in
			\mathbf{V}_{\gamma'},
		\end{equation*}
		there exists a unique $\pi\in \L^{\gamma}(\1)$, with $\int_{\1}\pi\,\dx=0$, such that
		\begin{equation*}
			\left\langle \uu^\ast,\uu\right\rangle =
			\int_{\1}\pi\, \div\uu\,\dx,\quad\text{ for all }\,\uu\in\W_{0}^{1,\gamma'}(\1)^d.
		\end{equation*}
		Moreover, there exists a positive constant $C$ such that
		\begin{equation*}
			\|\pi\|_{\gamma}\leq C \|\uu^\ast\|_{-1,\gamma}.
		\end{equation*}
\end{lemma}
\begin{proof}
	The proof is due to {Bogovski\v{\i}}~\cite{MEB} (see also \cite[Theorems~III.3.1 and III.5.3]{Galdi:2011} and \cite[Proposition I.1.1]{Te}).
\end{proof}

\section{Main results}\label{main:res}

{
We now introduce the notion of weak solutions to the problem \eqref{1.1}, which is the focus of our study.
At this stage, we impose the following minimal assumptions on the problem data to guarantee that the subsequent definition is meaningful,
\begin{equation}\label{hyp:def:ws}
\f\in\bfL^1(Q_T),\qquad  \vv_0\in\bfV,\qquad p>1.
\end{equation}
}

\begin{definition}\label{def.1}
{Assume that conditions \eqref{hyp:def:ws} are satisfied.}
We say that $\vv$ is a weak solution to the problem \eqref{1.1}, if:
\begin{enumerate}
  \item $\vv\in\L^\infty(0,T;\bfV)\cap\L^p(0,T;\bfV_p)$;
  \item ${\partial_t\vv}\in\L^2(0,T;\bfV)$;
  \item {$\vv(0)=\vv_0$;}
  \item The following identity,
\begin{equation}\label{3p3}
\begin{split}
& \hspace{-0.2cm}
-\!\int_{0}^{T}\!\!\!\!\int_{\Omega}\vv\cdot\partial_t\bfi \dxt - \kappa\!\int_{0}^{T}\!\!\!\!\int_{\Omega}\nabla\vv:\nabla\partial_t\bfi\dxt +
\nu\!\int_{0}^{T}\!\!\!\!\int_{\Omega}|\bfD(\vv)|^{p-2}\bfD(\vv):\bfD(\bfi)\dxt  \\
& \hspace{-0.2cm} =
\int_\1\vv_0\cdot\bfi(0)\dx + \kappa\int_\1\nabla\vv_0:\nabla\bfi(0)\dx + \!\int_{0}^{T}\!\!\!\!\int_{\Omega}\f\cdot\bfi\dxt + \!\int_{0}^{T}\!\!\!\!\int_{\Omega} \vv\otimes\vv:\mathds{\nabla}\bfi\dxt,
\end{split}
\end{equation}
holds for all $\bfi\in \bfC^\infty_{\div}(Q_T)$ with $\supp(\bfi) \Subset \Omega\times [0,T)$.
\end{enumerate}
\end{definition}

As usual, condition $\vv(0) = \vv_0$ is interpreted in the sense of weak continuity in time, that is
\begin{equation*}
  \lim_{t \to 0^+} \int_{\Omega} \vv(t) \cdot \bpsi = \int_{\Omega} \vv_0 \cdot \bpsi, \quad \text{ for all }\ \bpsi \in \bVcal.
\end{equation*}

The main result of this work concerns the existence of a weak solution to problem~\eqref{1.1}.
For clarity of presentation, the existence result is divided into two theorems. The first corresponds to the case where the theory of monotone operators is sufficient, while the second addresses the case in which additional analytical tools are required.

For the first existence result, we strength conditions in~\eqref{hyp:def:ws} on the forcing term and initial velocity as follows
\begin{alignat}{2}
\label{Hyp:u0}
  \vv_0&\in\bfV\cap\bfV_p,  \\
\label{Hyp:f:1}
  \f&\in\L^{\infty}(0,T;\W^{-1,2}(\1)^d)\cap\L^{p'}(0,T;\L^{p'}(\1)^d).
\end{alignat}

\begin{theorem}\label{thm:exist:1}
Let $\1\subset\R^d$ be a bounded domain with a \emph{Lipschitz-continuous boundary $\partial\1$}, and assume that hypotheses \eqref{Hyp:u0}--\eqref{Hyp:f:1} are verified.
Suppose that either
\begin{equation}\label{cond:1:thm1}
p>1\qquad \mbox{and}\qquad 2\leq d\leq 3,
\end{equation}
or
{
\begin{equation}\label{cond:2:thm1}
p\geq \frac{d}{2}\qquad \mbox{and}\qquad d>3.
\end{equation}
}
Then there exists, at least, a weak solution $\vv \in \C([0,T];\bfV)\cap \L^p(0,T;\bfV_p)$ to the problem \eqref{1.1}.
	\end{theorem}
{The second existence result fills the gap left by Theorem~\ref{thm:exist:1}, covering the remaining case $1<p<\tfrac{d}{2}$ with $d>3$, which is excluded by conditions~\eqref{cond:1:thm1} and~\eqref{cond:2:thm1}.
In that case, assumption~\eqref{Hyp:f:1} can be relaxed to}
\begin{equation}\label{Hyp:f}
  \f\in\L^{p'}(0,T;\L^{p'}(\1)^d).
\end{equation}
However, stronger assumptions on the regularity of the domain and on the spatial dimension will be required.

\begin{theorem}\label{thm:exist:2}
Let $\1\subset\R^d$ be a bounded domain with a \emph{smooth boundary $\partial\1$ of class $\C^2$}, and assume that hypotheses \eqref{Hyp:u0} and \eqref{Hyp:f} are verified.
If
{\begin{equation}\label{3.7}
1<p<\frac{d}{2}\qquad \mbox{and}\qquad d>3,
\end{equation}
}then there exists, at least, a weak solution $\vv \in \C([0,T];\bfV)\cap \L^p(0,T;\bfV_p)$ to the problem \eqref{1.1}.
	\end{theorem}

{Observe that condition $1<p<\tfrac{d}{2}$ in \eqref{3.7} is meaningful whenever $d>2$, and in particular for $d>3$.
It becomes especially significant when $d=4$, where it reduces to
\begin{equation*}
1 < p < 2 \qquad \mbox{and} \qquad d=4.
\end{equation*}
While physical applications are typically restricted to dimensions $d=2$ and $d=3$, the case $d=4$ is of analytical interest, as it lies at the borderline of the shear-thinning regime and poses mathematical challenges not encountered in lower dimensions.
}
Many existence and regularity results for nonlinear PDEs critically depend on the interplay between the spatial dimension and the growth of the nonlinearity.
Dimension $d=4$ represents a borderline dimension where the scaling of Sobolev embeddings and duality estimates is delicate.
Establishing existence in this dimension shows the robustness of the analytical framework and provides insight into the mechanisms that control solutions in higher-dimensional settings.
In this sense, Theorem~\ref{thm:exist:2} extends the applicability of the theory beyond physically relevant dimensions and highlights the structural stability of the approach under more demanding analytical conditions.

The next result addresses the uniqueness of weak solutions, showing that under the assumptions of Theorem~\ref{thm:exist:1}, or Theorem~\ref{thm:exist:2}, the problem \eqref{1.1} admits at most one solution in the sense of Definition~\ref{def.1}.

\begin{theorem}\label{UniT}
{	Under the assumptions of Theorems~\ref{thm:exist:1} and \ref{thm:exist:2}, for $p>1$ when $2\leq d\leq 4$, or for $p\geq \frac{d}{2}$ when $d>4$, the solution to the problem \eqref{1.1} is {unique in the class of weak solutions defined in Definition \ref{def.1}}.}
\end{theorem}

The proofs of Theorems~\ref{thm:exist:1} and \ref{thm:exist:2} are postponed to Sections~\ref{Sect:ex:thm:1} and \ref{Sect:ex:small}, whereas the proof of Theorem~\ref{UniT}  is carried out in Section~\ref{Sect:unique}.

\medskip
The proofs of Theorems~\ref{thm:exist:1} and \ref{thm:exist:2} are distinguished according to the cases \eqref{cond:1:thm1}--\eqref{cond:2:thm1} and \eqref{3.7}.
For Theorem~\ref{thm:exist:1}, the result follows directly by applying the standard theory of monotone operators.
In contrast, the proof of Theorem~\ref{thm:exist:2} requires, at the appropriate points, the introduction of technical tools and auxiliary results, which will be presented in the subsequent sections.

\section{Proof of Theorem~\ref{thm:exist:1}}\label{Sect:ex:thm:1}

{
The proof of Theorem~\ref{thm:exist:1} relies on the following embedding, which can be seen as a refinement of the classical embedding \eqref{imb-LL} used in the existence theory of weak solutions in~\cite{Lions:1969}:
\begin{equation}\label{imb:p:2}
\L^\infty(0,T; \W^{1,2}_0(\Omega)^d)
\cap
\L^\infty(0,T; \W^{1,p}_0(\Omega)^d)
\hookrightarrow
\L^{\rho}(0,T; \L^{\rho}(\Omega)^d),
\end{equation}
where
\begin{equation}\label{def:rho}
\rho := \max\left\{\frac{2d}{d-2},\frac{dp}{d-p}\right\}.
\end{equation}
This embedding follows from the classical Sobolev inequality combined with standard parabolic interpolation arguments.
The value of $\rho$ defined in \eqref{def:rho} is mainly relevant in the case $d>\max\{2,p\}$.
Indeed, when $d=2$ or $d=p$, the exponent $\rho$ can take arbitrarily large values in $[1,\infty)$, while if $d<p$ it is even possible to set $\rho = \infty$.
This reflects the fact that in these borderline or low-dimensional cases the embedding \eqref{imb:p:2} can be made as strong as needed, which greatly simplifies the control of the nonlinear terms.

On the other hand, the information provided solely by
$\L^\infty(0,T; \W^{1,2}_0(\Omega)^d)$ is sufficient to show that
\begin{equation}\label{map:L1}
t
\longmapsto
\int_\Omega
\boldsymbol{v}(t) \otimes \boldsymbol{v}(t)
:
\nabla \boldsymbol{\varphi}(t) \, \mathrm{d}x
\ \in \ \L^1(0,T),
\end{equation}
for all
$\boldsymbol{v},\boldsymbol{\varphi} \in \L^\infty(0,T; \W^{1,2}_0(\Omega)^d)$.
The key point is that \eqref{map:L1} does not impose any restriction on the exponent $p$, and thus this argument applies for all $p>1$.
Moreover, \eqref{map:L1} plays a central role in controlling the convective term, which allows the direct application of the theory of monotone operators to establish existence, in the spirit of~\cite{Ladyzhenskaya1969,Lions:1969}.
}

{
In the following, we present the proof of Theorem~\ref{thm:exist:1}, which is divided into several subsections. The proof is included for the sake of completeness and because some of its arguments will be used later in the proof of Theorem~\ref{thm:exist:2}.
}
\subsection{Approximate solutions}\label{Subs:GA}
We construct a solution to the problem \eqref{1.1} as a limit of suitable Galerkin approximations.
Let $s$ be the smallest positive integer such that
\begin{equation}\label{Ws2->W1infty:p}
\W^{s,2}_0(\1)^d\hookrightarrow \W^{1,\infty}_0(\1)^d,\qquad s>1+\frac{d}{2},
\end{equation}
and let us consider the space $\bfV_s$ associated to $\W^{s,2}_0(\1)^d$ defined in Section \ref{sec2}. 
		By means of separability, there exists a basis $\big\{\mathds{\bpsi}_{k}\big\}_{k\in\N}$ of $\bfV_s$, formed by the eigenfunctions of a suitable spectral problem, that is, orthogonal in $\L^2(\1)^d$ and that can be made orthonormal in $\W^{s,2}_0(\1)^d$ (see~\cite[Theorem~A.4.11]{Necas:book:1996}).
		Given
		$n\in\N$, let us consider the $n-$dimensional space $\bfX^{n}=\mathrm{span}\{\bpsi_{1}, \dots, \bpsi_{n}\}$.
		For each $n\in\N$, we search for approximate solutions {$\vv_n\in\C^1(0,T;\bfX^n)$} of the form
		\begin{align*}
			\vv_n(x,t)= \sum_{k=1}^nc_k^n(t)\bpsi_{k}(x),\quad\bpsi_{k}\in \bfX^n,
		\end{align*}
		where the coefficients $c_{1}^n(t),\dots,c_{n}^n(t)$
		are solutions of the following $n$ ordinary differential equations, derived from \eqref{3p3},
\begin{equation}\label{5.2:p}%
\begin{split}
&
\frac{\d}{\dt}\left[\big(\vv_n(t),\bpsi_{k}\big)+\kappa\big(\nabla\vv_n(t),\nabla\bpsi_{k}\big)\right] + \nu\big<|\bfD(\vv_n(t))|^{p-2}\bfD(\vv_n(t)),\bfD(\bpsi_{k})\big>
 \\
&=
\big(\f(t),\bpsi_{k}\big) + \big(\vv_n(t)\otimes\vv_n(t):\nabla\bpsi_{k}\big), \qquad k\in\{1,\dots,n\},
\end{split}
\end{equation}
		supplemented with the initial conditions
		\begin{align}\label{5.3:p}
			\vv_n(0)=\vv_{0,n},\quad\mbox{in}\ \1,
		\end{align}
		where {$\vv_{0,n}=P^n(\vv _{0})$}, and $P^n$ denotes the orthogonal projection
\begin{equation}\label{orth:proj:n}
P^n:\bfV\longrightarrow \bfX^n
\end{equation}
so that
		\begin{align*}
			\vv_n(0,x)=\sum_{k=1}^n c_k^n(0)\bpsi_k(x),\quad c_k^n(0)=c_{k,0}^n:=(\vv _{0},\bpsi_k),\quad k\in\{1,\dots,n\}.
		\end{align*}
By the uniform continuity of the orthogonal projection $P^n$, we may assume that
\begin{equation}\label{uniform:cont:p}
\vv_{0,n} \xrightarrow[n\to\infty]{} \vv_{0},\quad \text{ in }\ \W^{1,2}_0(\Omega)^d\cap\bfV_p.
\end{equation}
		Observe that the system (\ref{5.2:p})-(\ref{5.3:p}) can be written in the matrix form as follows,
		\begin{align}\label{5.4:p}
			\bfC \bfc'(t)=\bfb(t), \quad \bfc(0)=\bfc_{n,0},
		\end{align}
		where
		$\bfC=\left\{a^n_{\ell m}\right\}_{\ell,m}^n$, $\bfb(t)=\{b_{\ell}^n(t)\}_{\ell=1}^n$, $\bfc(t)=\{c_m^n(t)\}_{m=1}^n$,  with
		\begin{alignat*}{5}
			 a_{\ell m}^n&:=
			\left(\bpsi_{\ell},\bpsi_{m}\right)+\kappa\left(\nabla\bpsi_{\ell},\nabla\bpsi_{m}\right), \\
			 b_{\ell}^n(t)&:=
			\big(\f(t),\bpsi_{\ell}\big) + \big(\vv_n(t)\otimes\vv_n(t):\nabla\bpsi_{\ell}\big)
			-\nu\big<|\bfD(\vv_n(t))|^{p-2}\bfD(\vv_n(t)),\bfD(\bpsi_{\ell})\big>, \\
			 \bfc(t)&:=(c_{1}^n(t),\dots,c_{n}^n(t))\qquad\mbox{and}\qquad
			\bfc_{0}={(c_{1,0}^n,\dots,c_{n,0}^n)}.
		\end{alignat*}
Taking into account that the family $\big\{\bpsi_{k}\big\}_{k\in\N}$ is linearly independent in $\bfV$,
		$\big(\bfC\bxi,\bxi\big)>0$, for all $\bxi\in\R^n\setminus\{0\}$, we can write (\ref{5.4:p}) in the form
		\begin{align}\label{5.5:p}
			\bfc'(t)=\overline{\bfb(t)}, \quad \bfc(0)=\bfc_{n,0},
		\end{align}
where $\overline{\bfb(t)}=\bfC^{-1}\bfb(t)$.
For a globally Lipschitz-continuous coefficient function $\bfb(\cdot)$, we can use the Carathéodory theorem (see~\cite[Theorem~A.3.4]{Necas:book:1996}) to prove the existence of a solution $\bfc(t)$, in a short interval $[0,T^\ast]\subset[0,T]$, to the system \eqref{5.5:p}.

		\subsection{Uniform estimates}\label{uestimate}
We establish the uniform energy estimates, with respect to $n$, for the solution of the finite-dimensional approximated system corresponding to the system \eqref{1.1}.
		\begin{lemma}\label{thrmUE:p}
Assume $\f\in\L^{p'}(0,T;\L^{p'}(\1)^d)$.
\begin{enumerate}
  \item If $\vv_0\in\bfV$ and $p>1$, then there exists a positive constant $C_1$, independent of $n$, such that the following estimate holds true:
\begin{equation}\label{ae1:un:p}
\begin{split}
 &
\sup_{t\in[0,T]}\left\{\|\vv_n(t)\|^2_{2} + \kappa\|\nabla\vv_n(t)\|_{2}^2\right\} + \nu\int_0^T\|\nabla\vv_n(t)\|^p_{p}\dt
 \\
& \leq C_1\left(\int_0^T\|\f(t)\|^{p'}_{p'}\dt + \|\vv_0\|^2_{2} + \|\nabla\vv_0\|_{2}^2 \right);
\end{split}
\end{equation}
  \item Let $\vv_0 \in \bfV_p$, and assume that either hypothesis \eqref{cond:1:thm1} or \eqref{cond:2:thm1} is satisfied. Then, there exists a positive constant $C_2$, independent of $n$, such that the following estimate holds:
\begin{equation}\label{ae2:un':p}
\begin{split}
 &
 \sup_{t\in[0,T]}\|\nabla\vv_n(t)\|^p_{p} +
\int_0^T\left(\|{\partial_t\vv_{n}}(t)\|_{2}^2+\kappa\|\nabla{\partial_t\vv_{n}}(t)\|^2_{2}\right)\dt
\\
&
 \leq
 C_2\left(\int_0^T\|\f(t)\|^{p'}_{p'}\dt + \|\nabla\vv_0\|^p_{p} \right).
\end{split}
\end{equation}
\end{enumerate}
	\end{lemma}
\begin{proof}
To establish \eqref{ae1:un:p} and \eqref{ae2:un':p}, we multiply equation \eqref{5.2:p} by $c_k^n(t)$ and $\displaystyle \frac{\d c_k^n(t)}{\d t}$, respectively, and sum the resulting expressions over $k = 1$ to $k = n$.
The desired estimates then follow from suitable applications of H\"older's, Korn's, Sobolev's, Young's, and Gr\"onwall's inequalities.
The structure of the proof parallels the approach used in~\cite[Lemma~3.2]{Ant-Khom-2017}, though our momentum equation differs from the one considered there.
We just draw the reader's attention to statement~(2), where the distinction between the cases \eqref{cond:1:thm1} or \eqref{cond:2:thm1},
arises from the need to control the boundedness of the convective integral term
\begin{equation*}
\int_{\Omega} \big( \vv_n(t) \cdot \nabla \big) \vv_n(t) \cdot \partial_t \vv_n(t) \, \mathrm{d}x .
\end{equation*}
When $d = 2$ or $d = 3$, the Ladyzhenskaya inequalities (cf.~\cite[Chapter~1, Lemmas~1 and 2]{Ladyzhenskaya1969}) can be applied as follows:
\begin{equation*}
\begin{split}
\left| \int_{\Omega} \big( \vv_n(t) \cdot \nabla \big) \partial_t \vv_n(t) \cdot \vv_n(t) \, \mathrm{d}x \right|
& \leq \|\nabla \partial_t \vv_n(t)\|_{2} \, \|\vv_n(t)\|_{4}^2 \\
& \leq
\begin{cases}
\sqrt{2}\|\nabla \partial_t \vv_n(t)\|_{2}\|\vv_n(t)\|_{2}\|\nabla \vv_n(t)\|_{2}, & d = 2, \\[4pt]
2\|\nabla \partial_t \vv_n(t)\|_{2}\|\vv_n(t)\|_{2}^{\frac{1}{2}}\|\nabla \vv_n(t)\|_{2}^{\frac{3}{2}}, & d = 3.
\end{cases}
\end{split}
\end{equation*}
For $d > 3$, we use H\"older's inequality:
\begin{equation*}
\left| \int_{\Omega} \big( \vv_n(t) \cdot \nabla \big) \vv_n(t) \cdot \partial_t \vv_n(t) \, \mathrm{d}x \right|
\leq \|\vv_n(t)\|_{\frac{2d}{d-2}} \, \|\nabla \vv_n(t)\|_{p} \, \|\partial_t \vv_n(t)\|_{\frac{2d}{d-2}},
\end{equation*}
which holds provided that
\begin{equation*}
\frac{2}{\frac{2d}{d-2}} + \frac{1}{p} \leq 1
\quad \Leftrightarrow \quad p \geq \frac{d}{2}.
\end{equation*}
\end{proof}

The uniform estimate \eqref{ae1:un:p} enables us to use the Continuation Principle to extend the solution $\bfc(t)$ of the initial-value problem \eqref{5.5:p} to the whole interval $[0,T]$.

On the other hand, as a consequence of \eqref{imb:p:2}--\eqref{def:rho} and \eqref{ae1:un:p}, we have
\begin{equation}\label{un:b:uxu}
\int_0^T\|\vv_n(t)\otimes\vv_n(t)\|_{\frac{\rho}{2}}^{\frac{\rho}{2}}\dt\leq C, \qquad \rho=\max\left\{\frac{2d}{d-2},\frac{dp}{d-p}\right\},
\end{equation}
for some positive constant $C$ not depending on $n$.
Note that
\begin{equation*}
\rho=\left\{
\begin{array}{ll}
\displaystyle \frac{2d}{d-2}, & 1<p\leq 2, \\[10pt]
\displaystyle \frac{dp}{d-p}, & p\geq 2,
\end{array}
\right.
\end{equation*}
and thus
\begin{equation*}
\frac{\rho}{2}>1\Leftrightarrow
\left\{
\begin{array}{ll}
\displaystyle \frac{2d}{d-2}>2, & 1<p\leq 2, \\[10pt]
\displaystyle p>\frac{2d}{d+2}, & p\geq 2.
\end{array}
\right.
\end{equation*}
Therefore, the uniform estimate \eqref{un:b:uxu}, together with the fact that $\frac{\rho}{2}>1$, which will be used later (see \eqref{lim:uxu} below), does not impose any restriction on $p$.

We now show that
\begin{equation}\label{ub:comp}
\mbox{$\partial_t \vv_n$ \ is uniformly bounded in \ $\L^\infty(0,T;\W^{-s,2}(\Omega)^d)$},
\end{equation}
for $s$ satisfying \eqref{Ws2->W1infty:p}.
In fact, estimates \eqref{ae1:un:p} and \eqref{ae2:un':p} provide uniform bounds on $\vv_n$ in $\L^\infty(0,T; \W^{1,2}_0(\Omega)^d)$ and in $\L^\infty(0,T; \W^{1,p}_0(\Omega)^d)$, independently of $n$.
The Galerkin system \eqref{5.2:p} allows us to write
\begin{equation*}
(\I - \kappa \Delta) \partial_t \vv_n = P^n \Big( - \divv\!(\vv_n \otimes \vv_n) + \nu\, \divv\!(|\mathbf{D}(\vv_n)|^{p-2} \mathbf{D}(\vv_n)) + \f \Big),
\end{equation*}
where $P^n$ is the orthogonal projection defined in \eqref{orth:proj:n}.
As $\kappa>0$, the elliptic operator $(\I - \kappa \Delta)$ is invertible, and, because $P^n$ acts on a finite-dimensional space, $(\I - \kappa \Delta)^{-1} P^n$ is bounded.
Therefore it is sufficient to estimate
\begin{equation}\label{Gal:n:dt:1}
\Big( - \divv\!(\vv_n \otimes \vv_n) + \nu\,\divv\!(|\mathbf{D}(\vv_n)|^{p-2} \mathbf{D}(\vv_n)) + \f \Big)
\end{equation}
in $\L^\infty(0,T;\W^{-s,2}(\Omega)^d)$ to control $\partial_t \vv_n$.
The convective term can be estimated, for any
$\boldsymbol{\phi} \in \W_0^{s,2}(\Omega)^d$, as
\begin{equation}\label{Gal:n:dt:2}
\begin{split}
\left|\big\langle - \divv\!(\vv_n(t) \otimes \vv_n(t)), \boldsymbol{\phi} \big\rangle\right|
= & \left|\int_\Omega \vv_n(t) \otimes \vv_n(t):\nabla\boldsymbol{\phi} \,\mathrm{d}x \right|
\leq \|\vv_n(t)\|_{2^\ast}^2\|\nabla\boldsymbol{\phi}\|_{\infty} \\
\leq  & C\|\nabla\vv_n(t)\|_{2}^2\|\boldsymbol{\phi}\|_{s,2},
\end{split}
\end{equation}
for  $s>1+\frac{d}{2}$.
For the viscous term, we also have for any
$\boldsymbol{\phi} \in \W_0^{s,2}(\Omega)^d$
\begin{equation}\label{Gal:n:dt:3}
\begin{split}
&\left|\big\langle\divv\,\left(|\mathbf{D}(\vv_n(t))|^{p-2} \mathbf{D}(\vv_n(t))\right), \boldsymbol{\phi} \big\rangle\right|
 \\&= \left|\int_\Omega |\mathbf{D}(\vv_n(t))|^{p-2} \mathbf{D}(\vv_n(t))):\nabla\boldsymbol{\phi} \,\mathrm{d}x \right| \\
&\leq  \|\mathbf{D}(\vv_n(t))\|_{p-1}^{p-1}\|\nabla\boldsymbol{\phi}\|_{\infty}
\leq  C\|\mathbf{D}(\vv_n(t))\|_{p}^{p-1}\|\boldsymbol{\phi}\|_{s,2},
\end{split}
\end{equation}
for  $s>1+\frac{d}{2}$ as well.
As for the forcing term, we invoke assumption \eqref{Hyp:f:1}, in particular the fact that $\f \in \L^{\infty}(0,T;\W^{-1,2}(\Omega)^d)$. Moreover, the continuous embedding $\W^{-1,2}(\Omega)^d \hookrightarrow \W^{-s,2}(\Omega)^d$, valid for every $s>1$, yields the same bound in the larger space. Consequently, we also have
\begin{equation}\label{Gal:n:dt:4}
\f \in \L^{\infty}(0,T;\W^{-s,2}(\Omega)^d).
\end{equation}
Combining \eqref{Gal:n:dt:1}--\eqref{Gal:n:dt:4} with the estimates \eqref{ae1:un:p} and \eqref{ae2:un':p}, we obtain \eqref{ub:comp}.

We can now apply the Aubin-Dubinski\u{\i} lemma with the compact embedding
$\W^{1,2}_0(\Omega)^d \hookrightarrow\hookrightarrow \L^2(\Omega)^d$
and the continuous embedding
$\L^2(\Omega)^d \hookrightarrow \W^{-s,2}(\Omega)^d$, valid for every for $s>0$.
Since, by \eqref{ae1:un:p}, $\vv_n$ is uniformly bounded in
$\L^\infty(0,T;\W^{1,2}_0(\Omega)^d)$ and, by \eqref{ub:comp},
$\partial_t \vv_n$ is uniformly bounded in
$\L^\infty(0,T;\W^{-s,2}(\Omega)^d)$, the Aubin-Dubinski\u{\i} lemma implies that,
up to a subsequence,
\begin{equation}\label{str:con:C0T}
\vv_n \xrightarrow[n\to\infty]{} \vv,\quad \text{ in }\ \C([0,T];\L^2(\Omega)^d).
\end{equation}

\subsection{Passing to the limit as $n\to\infty$}\label{WC}
		In view of \eqref{ae1:un:p} and \eqref{ae2:un':p}, the Banach--Alaoglu theorem guarantees the existence of functions $\vv$ and $\bmcA$ such that, for some subsequences (still denoted by the same indices), we have:
\begin{alignat}{2}
	\label{5.10:p}
	\vv_n &\xrightharpoonup[n\to\infty]{\ast} \vv, \quad && \text{in} \quad \L^\infty(0,T;\bfV), \\
	\label{5.10:1:p}
	\vv_n &\xrightharpoonup[n\to\infty]{\ast} \vv, \quad && \text{in} \quad \L^\infty(0,T;\bfV_p), \\
	\label{5.11:p}
	\vv_n &\xrightharpoonup[n\to\infty]{} \vv, \quad && \text{in} \quad \L^{p}(0,T;\bfV_p), \\
	\label{5.012:p}
	\bfA(\vv_n) &\xrightharpoonup[n\to\infty]{} \bmcA, \quad && \text{in} \quad \L^{p'}(0,T;\L^{p'}(\1)^{d \times d}), \\
	\label{5.14:p}
	\partial_t \vv_n &\xrightharpoonup[n\to\infty]{} \partial_t \vv, \quad && \text{in} \quad \L^2(0,T;\bfV).
\end{alignat}

{
On the other hand, \eqref{str:con:C0T} implies that, for some subsequence not relabeled,
\begin{equation}\label{str:conv:Ls}
	\vv_n \xrightarrow[n \to \infty]{} \vv,
	\quad \text{in} \quad \L^2(0,T;\L^2(\Omega)^d).
\end{equation}
}
By the Riesz--Fischer theorem, there exists a subsequence, still denoted by the same index, such that
\begin{equation*}
	\vv_n \xrightarrow[n \to \infty]{} \vv, \quad \text{a.e. in } Q_T,
\end{equation*}
and consequently,
\begin{equation}\label{5.42:1:a.e.:p}
	\vv_n \otimes \vv_n \xrightarrow[n \to \infty]{} \vv \otimes \vv, \quad \text{a.e. in } Q_T.
\end{equation}
Now, in view of \eqref{un:b:uxu} and \eqref{5.42:1:a.e.:p}, a classical result (see e.g.~\cite[Lemma~1.3]{Lions:1969}) yields
\begin{equation}\label{lim:uxu}
	\vv_n \otimes \vv_n \xrightharpoonup[n \to \infty]{} \vv \otimes \vv, \quad \text{in} \quad \L^{\frac{\rho}{2}}(0,T;\L^{\frac{\rho}{2}}(\1)^{d \times d}).
\end{equation}

\subsection{Application of the Minty trick}\label{MT}

Multiplying \eqref{5.2:p} by $c_k^n\in\C^1(0,T)$, with $k\in\{1,\dots,n\}$, and adding up the resulting equation from $k=1$ to $n$, and integrating over $(0,T)$, one has
\begin{equation}\label{5.2:n:p}%
\begin{split}
&
\!\int_{0}^{T}\!\!\!\!\int_{\Omega}\big({\partial_t\vv_{n}}\cdot\bpsi_j  + \kappa\nabla{\partial_t\vv_{n}}:\nabla\bpsi_j\big)\dxt +
\nu\!\int_{0}^{T}\!\!\!\!\int_{\Omega}|\bfD(\vv_n)|^{p-2}\bfD(\vv_n):\bfD(\bpsi_j)\dxt\\
& =
\!\int_{0}^{T}\!\!\!\!\int_{\Omega}\f\cdot\bpsi_j\dxt + \!\int_{0}^{T}\!\!\!\!\int_{\Omega}\vv_n\otimes\vv_n:\nabla\bpsi_j\dxt,
\end{split}
\end{equation}
for any function
\begin{equation*}
\bpsi_j(\x,t)=\sum\limits_{k=1}^j{c_k^j(t)}\bpsi_{k}(x),\qquad \bpsi_{k}\in \bfX^j,\quad k\in\{1,\dots,j\},\quad  j\in\{1,\dots,n\}.
\end{equation*}
In particular, taking $\bpsi_j=\vv_n$ in \eqref{5.2:n:p}, we have
\begin{equation}\label{5.2:un:p}
  \begin{split}
  & \|\vv_n(T)\|^2_{2}+\kappa\|\nabla\vv_n(T)\|_{2}^2+2\nu \!\int_{0}^{T}\!\!\!\!\int_{\Omega}|\bfD(\vv_n)|^p \dxt\\ &=
   \|\vv_{0,n}\|^2_{2}+\kappa\|\nabla\vv_{0,n}\|_{2}^2+2\!\int_{0}^{T}\!\!\!\!\int_{\Omega}\f\cdot\vv_n \dxt.
  \end{split}
\end{equation}

On the other hand, using the convergence results \eqref{5.10:p}--\eqref{5.14:p} and \eqref{lim:uxu},  we can pass to the limit as $n\to\infty$ in \eqref{5.2:n:p} so that
\begin{equation}\label{5.2:lim:p}%
\begin{split}
&
\!\int_{0}^{T}\!\!\!\!\int_{\Omega}\big({\partial_t\vv}\cdot\bpsi_j  + \kappa\nabla{\partial_t\vv}:\nabla\bpsi_j\big)\dxt +
\nu\!\int_{0}^{T}\!\!\!\!\int_{\Omega}\bmcA:\nabla\bpsi_j\dxt \\
&=
\!\int_{0}^{T}\!\!\!\!\int_{\Omega}\f\cdot\bpsi_j\dxt + \!\int_{0}^{T}\!\!\!\!\int_{\Omega}\vv\otimes\vv:\nabla\bpsi_j\dxt,
\end{split}
\end{equation}
for any such function $\bpsi_j$.
Since the functions $\left\{\bpsi_j\right\}_{j\in\N}$ are dense in $\L^\infty(0,T;\bfV)\cap\L^p(0,T;\bfV_p)$, we can put $\bpsi_j=\vv$ in \eqref{5.2:lim:p}, which yields
\begin{equation}\label{5.2:lim:u:p}
  \begin{split}
  & \|\vv(T)\|^2_{2}+\kappa\|\nabla\vv(T)\|_{2}^2+2\nu \!\int_{0}^{T}\!\!\!\!\int_{\Omega} \bmcA:\nabla\vv \dxt
  \\
  & =
  \|\vv_0\|^2_{2}+\kappa\|\nabla\vv_0\|_{2}^2+2\!\int_{0}^{T}\!\!\!\!\int_{\Omega}\f\cdot\vv \dxt.
  \end{split}
\end{equation}

Now, we subtract \eqref{5.2:lim:u:p} from \eqref{5.2:un:p}, and thus we have
\begin{equation}\label{5.2:un:u:p}
\begin{split}
&   \!\int_{0}^{T}\!\!\!\!\int_{\Omega}|\bfD(\vv_n)|^{p}\dxt-\!\int_{0}^{T}\!\!\!\!\int_{\Omega}\bmcA:\bfD(\vv) \dxt =: \frac{1}{2\nu}\big(I_1+I_2+I_3\big),
\end{split}
\end{equation}
where
\begin{alignat*}{2}
& I_1:=  \|\vv(T)\|^2_{2}+\kappa\|\nabla\vv(T)\|_{2}^2-\big(\|\vv_n(T)\|^2_{2}+\kappa\|\nabla\vv_n(T)\|_{2}^2\big), && \\
& I_2:=  \|\vv_{0,n}\|^2_{2}+\kappa\|\nabla\vv_{0,n}\|_{2}^2- \big(\|\vv_0\|^2_{2}+\kappa\|\nabla\vv_0\|_{2}^2\big), && \\
& I_3:= 2\!\int_{0}^{T}\!\!\!\!\int_{\Omega}\f\cdot\big(\vv_n-\vv\big) \dxt.
\end{alignat*}
By the monotonicity of the operator $\bfA(\vv)=|\bfD(\vv)|^{p-2}\bfD(\vv)$ (see \eqref{2.4}), there holds
\begin{equation}\label{mon:A(xi):p}
\begin{split}
 \!\int_{0}^{T}\!\!\!\!\int_{\Omega}|\bfD(\vv_n)|^{p}\dxt& \geq
\!\int_{0}^{T}\!\!\!\!\int_{\Omega}|\bfD(\bpsi)|^{p-2}\bfD(\bpsi):\bfD(\vv_n-\bpsi) \dxt \\&\quad + \!\int_{0}^{T}\!\!\!\!\int_{\Omega}|\bfD(\vv_n)|^{p-2}\bfD(\vv_n):\bfD(\bpsi)\dxt,
\end{split}
\end{equation}
for any function $\bpsi\in\overline{\mathrm{span}\{\bpsi_1,\dots,\bpsi_{n}\}}$.
Then, combining \eqref{5.2:un:u:p} with \eqref{mon:A(xi):p}, we have
\begin{equation}\label{5.2:un:1:p}
\begin{split}
& \!\int_{0}^{T}\!\!\!\!\int_{\Omega}|\bfD(\bpsi)|^{p-2}\bfD(\bpsi):\bfD(\vv_n-\bpsi) \dxt +  \!\int_{0}^{T}\!\!\!\!\int_{\Omega}|\bfD(\vv_n)|^{p-2}\bfD(\vv_n):\bfD(\bpsi)\dxt  \\&\qquad
- \!\int_{0}^{T}\!\!\!\!\int_{\Omega}\bmcA:\bfD(\vv) \dxt\\& \leq \frac{1}{2\nu}\left(I_1+I_2+I_3\right).
\end{split}
\end{equation}
Using the convergence results \eqref{5.11:p} and \eqref{5.012:p}, we can pass \eqref{5.2:un:1:p} to the limit, as $n\to\infty$, so that
\begin{equation}\label{5.2:un:2:p}
\begin{split}
& \!\int_{0}^{T}\!\!\!\!\int_{\Omega}\big(|\bfD(\bpsi)|^{p-2}\bfD(\bpsi)-\bmcA\big):\bfD(\vv-\bpsi)
\leq C\lim_{n\to\infty}\left(I_1+I_2+I_3\right),
\end{split}
\end{equation}
for some positive constant $C=C(\nu)$.
By \eqref{5.10:p}, the properties of $\limsup$ and $\liminf$, and the lower semicontinuity of the norm, one has
\begin{equation*}
\lim\limits_{n\to\infty}I_1\leq\limsup\limits_{n\to\infty}I_1\leq
{\|\vv(T)\|^2_{2}+\kappa\|\nabla\vv(T)\|_{2}^2}-\liminf\limits_{n\to\infty}\big(\|\vv_n(T)\|^2_{2}+\kappa\|\nabla\vv_n(T)\|_{2}^2\big)\leq 0.
\end{equation*}
By construction (see~\eqref{5.3:p}--\eqref{uniform:cont:p}),
\begin{equation*}
\lim\limits_{n\to\infty}I_2= 0.
\end{equation*}
Due to assumption \eqref{Hyp:f} and \eqref{5.11:p},
\begin{equation*}
\lim\limits_{n\to\infty}I_3= 0.
\end{equation*}
The above convergences allow us to write \eqref{5.2:un:2:p} as follows,
\begin{equation}\label{5.2:un:3:p}
\begin{split}
& \!\int_{0}^{T}\!\!\!\!\int_{\Omega}\big(|\bfD(\bpsi)|^{p-2}\bfD(\bpsi)-\bmcA\big):\bfD(\vv-\bpsi) \dxt
\leq 0.
\end{split}
\end{equation}
Taking $\bpsi=\vv\pm\delta\uu$, with $\delta>0$ and $\uu\in\L^\infty(0,T;\bfV)\cap\L^p(0,T;\bfV_p)$, in \eqref{5.2:un:3:p}, we obtain
\begin{equation*}
\begin{split}
& \mp\delta\!\int_{0}^{T}\!\!\!\!\int_{\Omega}\big(|\bfD(\vv\pm\delta\uu)|^{p-2}\bfD(\vv\pm\delta\uu)-\bmcA\big):\bfD(\uu) \dxt
\leq 0.
\end{split}
\end{equation*}
Letting $\delta\longrightarrow 0$, we see that it must be
\begin{equation*}
\mp\!\int_{0}^{T}\!\!\!\!\int_{\Omega}\big(|\bfD(\vv)|^{p-2}\bfD(\vv)-\bmcA\big):\bfD(\uu) \dxt= 0,
\end{equation*}
 for all $\uu\in\L^\infty(0,T;\bfV)\cap\L^p(0,T;\bfV_p)$.
Hence
\begin{equation}\label{A:lim}
\bmcA=|\bfD(\vv)|^{p-2}\bfD(\vv),\quad \mbox{a.e. in}\ \ Q_T.
\end{equation}

\subsection{Conclusion}
We can see that (1)-(2) of Definition~\ref{def.1} are easily verified due to \eqref{5.10:p}, \eqref{5.11:p} and \eqref{5.14:p}.
On the other hand, multiplying \eqref{5.2:p} by $\psi\in\C^\infty(0,T)$, with $\psi(T)=0$, integrating over $(0,T)$, using integration by parts in time, using \eqref{5.3:p}, and the convergence results \eqref{uniform:cont:p}, \eqref{5.10:p},  \eqref{5.012:p}, \eqref{lim:uxu} (observing here \eqref{map:L1} for the choice of $\rho$ in \eqref{imb:p:2}) and \eqref{A:lim}, we obtain \eqref{3p3} for any {$\bfi=\psi\bpsi_{k}$}, with $\bpsi_{k}\in \bfX^n$, $\psi\in\C^\infty(0,T)$, and $\psi(T)=0$. Thus, by continuity and density, \eqref{3p3} holds true for any $\bfi\in \bfC^\infty_{0,\div}(Q_T)$.
Finally, since $\partial_t\vv \in \L^2(0,T;\bfV)$ and $\vv\in \L^2(0,T;\bfV)$, we have
$\vv \in \W^{1,2}(0,T;\bfV) \subset \C([0,T];\bfV)$.
Consequently, the initial condition $\vv(0)=\vv_0$, stated in (3) of Definition~\ref{def.1}, is meaningful.
This concludes the proof of Theorem~\ref{thm:exist:1}.

\medskip
{
In the next two sections of this work, we focus on the most challenging case (see Theorem~\ref{thm:exist:2} and \eqref{3.7} therein):
{\begin{equation*}
1<p<\frac{d}{2}\qquad \mbox{and}\qquad d>3.
\end{equation*}
}
This intermediate range of the power-law exponent $p$ is significantly more delicate in spatial dimensions $d>3$, because the embedding in~\eqref{imb:p:2} fails to hold in this setting.
In lower dimensions, this embedding is valid, and the associated estimates can be carried out straightforwardly.
However, for $d>3$, additional analytical tools are required to handle the nonlinear terms. In particular, it becomes necessary to perform a refined decomposition of the pressure into multiple components in order to compensate for the loss of compactness and to secure the weak convergence properties essential to the existence argument.
The next section is devoted to developing the technical aspects of this pressure decomposition, which will play a crucial role in the construction of weak solutions for the case of $p$ satisfying \eqref{3.7}.
}

\section{Revisiting the pressure decomposition technique}\label{PD}
In this section, we look at the pressure term that appears in our model and will decompose it so that each part matches a specific term in the momentum equation.
The time derivative (inertia) and Voigt terms are grouped into one part of the pressure, while the viscous and convective terms are linked to two other pressure parts that are similar but have different roles.
The key idea behind the following result was originally introduced in~\cite{Wolf:2007} and later used in~\cite{DRW}, building on the original insight.
{Throughout this section, let $G \subset \mathbb{R}^d$ be a bounded domain with $\partial G \in C^2$.
We further assume that the spaces in \eqref{V-cal}--\eqref{V-p:k} are defined over $G$.} 
\begin{theorem}\label{thrm41}
{Consider $\vv \in \C_{\rw}([0,T]; \bfV)$, and, for $m_1,\, m_2\in(1,\infty)$, let
\begin{eqnarray*}
  && \bcG_1 \in \L^{m_1}(0,T;\L^{m_1}(G)^{d\times d}), \\
  && \bcG_2 \in \L^{m_2}(0,T;\L^{m_2}(G)^{d\times d}),\quad \mbox{with}\ \ \divv\bcG_2 \in \L^{m_2}(0,T;\L^{m_2}(G)^{d}).
\end{eqnarray*}

Assume further that condition \eqref{Hyp:u0} holds.
In addition, suppose the following identity is satisfied,
\begin{equation}\label{PD1}
\begin{split}
   & -\!\int_0^T \!\!\!\int_G \vv \cdot \partial_t \bfi \, \dxt
  - \kappa \!\int_0^T \!\!\!\int_G \nabla \vv : \nabla \partial_t \bfi \, \dxt +
   \!\int_0^T \!\!\!\int_G (\bcG_1 + \bcG_2) : \nabla \bfi \, \dxt=0,
\end{split}
\end{equation}
for all $\bfi \in \bfC^\infty_{0}(Q_T)$ with $\div\bfi=0$.

Then, there exist $\pi_1 \in \L^{m_1}(0,T; \L^{m_1}(G))$ and $\pi_2 \in \L^{m_2}(0,T; \W^{1,m_2}(G))$,  and there exists $\pi_h \in \C_{\rw}([0,T]; \W^{1,2}(G))$ such that
\begin{equation}\label{decomp:press}
\begin{split}
&\hspace{-2mm} -\!\int_0^T \!\!\!\int_G \vv \cdot \partial_t \bfi \, \dxt
  - \kappa \!\int_0^T \!\!\!\int_G \nabla \vv : \nabla \partial_t \bfi \, \dxt + \!\int_0^T \!\!\!\int_G (\bcG_1 + \bcG_2) : \nabla \bfi \, \dxt
  \\
  & =
  \!\int_0^T \!\!\!\int_G (\pi_1 + \pi_2) \, \div\bfi \, \dxt + \!\int_0^T \!\!\!\int_G \nabla \pi_h \cdot \partial_t \bfi \, \dxt \\
&\quad
  + \!\int_0^T \!\!\!\int_G \big( \vv(0) \cdot \bfi(0)
  + \kappa \nabla \vv(0) : \nabla \bfi(0) \big) \, \dxt,
\end{split}
\end{equation}
for all $\bfi \in \bfC^\infty(Q_T)$ with $\supp(\bfi)\Subset G \times[0,T)$.

In addition, we have
\begin{align}\label{press:3:1}
-\Delta \pi_h&=0,\quad \pi_h(0)=0, \\
 \|\pi_h(t)\|_{1,2} &\leq C_h\big(\|\vv(t)-\vv(0)\|_2+\kappa\|\nabla\vv(t)-\nabla\vv(0)\|_2\big), \label{press:3:2} \\
 \int_{0}^{T}\|\pi_1(t)\|_{m_1}^{m_1}\d t &\leq C_1\int_{0}^{T}\|\bcG_1(t)\|_{m_1}^{m_1}\d t, \label{press:3:3:1} \\
 \int_{0}^{T}\|\pi_2(t)\|_{m_2}^{m_2}\dt &\leq C_2\int_{0}^{T}\|\bcG_2(t)\|_{m_2}^{m_2}\d t, \label{press:3:3:2} \\
  \int_{0}^{T}\|\nabla\pi_2(t)\|_{m_2}^{m_2}\d t &\leq C_3\int_{0}^{T}\left(\|\bcG_2(t)\|_{m_2}^{m_2}+\|\divv\bcG_2(t)\|_{m_2}^{m_2}\right)\d t, \label{press:3:4}
\end{align}
for some positive constants $C_h=C(d,G)$ and $C_i=C(m_i,d,G)$ for $i\in\{1,2,3\}$.

}

\end{theorem}
It should be noted that if we combine all the components of the decomposed pressure, we obtain
\begin{align*}
    \pi(t) = \pi_h(t) + \int_0^t \pi_{1}(s)\, \mathrm{d}s+\int_0^t \pi_{2}(s)\, \mathrm{d}s,
\end{align*}
which implies that $\pi \in \L^\infty(0,T; \L^\gamma(G))$.

\begin{proof}[Proof of Theorem \ref{thrm41}]
{
	The proof has been divided into {several steps}.
	
	\vspace{2mm}
	\noindent\textbf{Step 1:}
Let us define a bilinear form $\mathcal{Z}$ on the domain $\mathcal{D}(\mathcal{Z}) := \W_0^{1,2}(G)^d \subset \L^2(G)^d$ by
\begin{align*}
    \mathcal{Z}(\vv,\uu) := \int_{G} \vv \cdot \uu \, \mathrm{d}\x + \kappa \int_G \nabla \vv : \nabla \uu \, \mathrm{d}\x,
\end{align*}
where $\kappa > 0$ is the relaxation time constant in the Navier-Stokes-Voigt model.

We begin by observing that $\mathcal{Z}$ is well-defined and bilinear on $\W_0^{1,2}(G)^d$, since both the $\L^2$-inner product and the (Frobenius) inner product of gradients are well-defined for functions in $\W_0^{1,2}(G)^d$.

Next, we consider the norm induced by this bilinear form, given by
\begin{align*}
    \mathcal{Z}(\vv,\vv)^{\frac{1}{2}} = \left( \|\vv\|_{2}^2 + \kappa \|\nabla \vv\|_{2}^2 \right)^{\frac{1}{2}}.
\end{align*}
This defines a norm on $\W_0^{1,2}(G)^d$ equivalent to the standard $\W^{1,2}$-norm, due to the Poincaré inequality (which applies because of the zero boundary conditions). Hence, $\W_0^{1,2}(G)^d$ is complete with respect to the $\mathcal{Z}$-norm. This completeness implies that the bilinear form $\mathcal{Z}$ is closed. Moreover, it follows immediately from its definition that $\mathcal{Z}$ is a symmetric and positive form.

Given that $\mathcal{Z}$ is a densely defined, closed, symmetric, and positive bilinear form on the Hilbert space $\L^2(G)^d$, we can apply \cite[Lemma II.3.2.1]{HS} to conclude the existence of a unique positive self-adjoint operator $\mathrm{A}: \D(\mathrm{A}) \to \L^2(G)^d$, with dense domain $\D(\mathrm{A}) \subset \W_0^{1,2}(G)^d$, such that
\begin{align*}
    \int_G \mathrm{A}(\vv) \cdot \uu \, \mathrm{d}\x = \mathcal{Z}(\vv,\uu)
    = \int_{G} \vv \cdot \uu \, \mathrm{d}\x + \kappa \int_G \nabla \vv : \nabla \uu \, \mathrm{d}\x,
\end{align*}for all $\vv \in \D(\mathrm{A}),\ \uu \in \W_0^{1,2}(G)^d$.

Fixing an arbitrary test function $\uu \in \C_0^\infty(G)^d$, we observe that the relation
\begin{equation*}
\mathrm{A}(\vv) = (\I - \kappa \Delta)(\vv) = (\I - \kappa\, \divv \nabla)(\vv)
\end{equation*}
holds in the sense of distributions. Based on this (formal) identity, we define the operator $\mathrm{A} := \I - \kappa \Delta$. Thus, the operator $\mathrm{A} = \I - \kappa \Delta : \D(\I - \kappa \Delta) \to \L^2(G)^d$ is characterized by the relation
\begin{align} \label{PD3}
    \int_G (\I - \kappa \Delta)(\vv) \cdot \uu \, \mathrm{d}\x
    = \mathcal{Z}(\vv, \uu)
    = \int_G \vv \cdot \uu \, \mathrm{d}\x + \kappa \int_G \nabla \vv : \nabla \uu \, \mathrm{d}\x,
\end{align}
for all $\vv \in \D(\I - \kappa \Delta)$ and $\uu \in \W_0^{1,2}(G)^d$.
The domain of $\mathrm{A}$ is given by the set
\begin{align*}
    \D(\I - \kappa \Delta)
    = \left\{ \vv \in \W_0^{1,2}(G)^d : \uu \mapsto \mathcal{Z}(\vv, \uu) \text{ is continuous with respect to the } \L^2 \text{-norm} \right\}.
\end{align*}
That is, $\vv \in \D(\I - \kappa \Delta)$ if and only if the linear functional $\mathcal{Z}(\vv, \cdot)$ extends continuously to all of $ \L^2(G)^d$, meaning there exists a function $\mathrm{A}(\vv) \in \L^2(G)^d$ such that the identity~\eqref{PD3} holds for all $\uu \in \W_0^{1,2}(G)^d$.

We can now reformulate equation~\eqref{PD1} into a weak form that is well-suited for the pressure decomposition we aim to develop. Specifically, making use of identity \eqref{PD3}, we consider \eqref{PD1} in the following form:
\begin{equation} \label{PD4}
\begin{split}
   -\!\int_0^T \!\!\int_G (\I - \kappa \Delta)\vv \cdot \partial_t \bfi \, \dxt
   + \int_0^T \!\!\int_G (\bcG_1 + \bcG_2) : \nabla \bfi \, \dxt=0,
\end{split}
\end{equation}
for all $\bfi \in \bfC^\infty_0(Q_T)$ with $\div\bfi=0$.
This formulation sets the stage for decomposing the pressure into distinct components associated with different terms in the momentum equation.

Using \eqref{PD4} in conjunction with Lemma \ref{lemm:BP}, and proceeding as in the proof of \cite[Theorem 2.6]{Wolf:2007}, we obtain the existence of a (unique) pressure $\pi\in \C_{\mathrm{w}}([0,T];\L^1( G ))$, with
\begin{align}\label{PD5}
	\int_G \pi(t)\d x=0,\quad \text{ for all }\ t\in[0,T],
\end{align}
and such that
\begin{align}\label{PD6}
&	\int_G (\I-\kappa\Delta) \big(\vv(t)-\vv(0)\big)\cdot \bpsi \d \x  + \int_G\big(\widetilde{\bcG}_1(t)+\widetilde{\bcG}_2(t)\big):\nabla \bpsi \d \x =
\int_G \pi(t)\div\bpsi \d \x,
\end{align}
for all $\bpsi \in \W_{0}^{1,2}(G)^d$ and for all $t\in[0,T]$, where
\begin{equation}\label{BcG:i}
\widetilde{\bcG}_i(t):=\int_{0}^{t}\bcG_i(s)\ds,\quad t\in(0,T),\quad i\in\{1,2\}.
\end{equation}

\vspace{2mm}
\noindent
\textbf{Step 2:}
Next, for $i \in \{1,2\}$ and $t \in [0,T]$, let $\uu_i(t) \in \bfV_{m_i}$ be the unique solution (cf.~\cite[Theorem~IV.6.1]{Galdi:2011}) to the following Stokes problem,
\begin{equation}
	\left\{\begin{aligned}
   -\Delta\uu_i(t)+\nabla\widetilde{\pi}_i(t) &= \divv\widetilde{\bcG}_i(t), \qquad &&\text{in } G, \\
   \div\uu_i(t) &= 0, \qquad &&\text{in } G, \\
  \uu_i(t) &= 0, \qquad &&\text{on } \partial G.
\end{aligned}
\right.
\end{equation}
In view of Lemma~\ref{lemm:BP}, for each $i \in \{1,2\}$,
\begin{equation}\label{prop:pi:i}
  \widetilde{\pi}_i(t) \in \L^{m_i}(G), \text{ and }  \int_G \widetilde{\pi}_i(t)\, \mathrm{d}x = 0, \quad \text{ for all }\ t \in [0,T].
\end{equation}
Proceeding as in the proof of \cite[Theorem~2.2]{DRW}, we show that, in fact, for each $i \in \{1,2\}$,
\begin{equation*}
\widetilde{\pi}_i \in \L^{m_i}(0,T; \L^{m_i}(G)),
\end{equation*}
and
\begin{equation*}
\int_{0}^{T} \|\partial_t \widetilde{\pi}_i(t)\|_{m_i}^{m_i} \, \mathrm{d}t \leq C_i \int_{0}^{T} \|\bcG_i(t)\|_{m_i}^{m_i} \, \mathrm{d}t.
\end{equation*}
for some positive constants $C_i$, $i\in\{1,2\}$.
Setting
\begin{equation}\label{der:pi}
  \pi_i := \partial_t \widetilde{\pi}_i, \quad i \in \{1,2\},
\end{equation}
we immediately obtain \eqref{press:3:3:1}--\eqref{press:3:3:2}.
The proof of \eqref{press:3:4} is considerably longer, so we refer the reader to the proof of \cite[Lemma~2.3]{Wolf:2007}, where it is presented in full detail.

Now, for $t \in [0,T]$, let $\uu_h(t) \in \bfV \cap \W^{2,2}(G)^d$ be the unique solution (cf.~\cite[Theorem~IV.6.1]{Galdi:2011}) to the following Stokes problem,
\begin{equation}
\left\{\begin{aligned}
   -\Delta \uu_h(t) + \nabla \pi_h(t) &= -(\I - \kappa \Delta) \big(\vv(t) - \vv(0)\big), \qquad &&\text{in } G, \\
  \div\uu_h(t)& = 0, \qquad &&\text{in } G, \\
   \uu_h(t) &= 0, \qquad &&\text{on } \partial G.
\end{aligned}
\right.
\end{equation}
Moreover,
\begin{alignat}{2}
 \int_G \pi_h(t)\, \mathrm{d}x &= 0, \label{prop:pi:h} \\ \nonumber
 \|\pi_h(t)\|_{1,2} &\leq C_h \|(\I - \kappa \Delta)(\vv(t) - \vv(0))\|_2\\& \leq
C_h \left( \|\vv(t) - \vv(0)\|_2 + \kappa \|\nabla \vv(t) - \nabla \vv(0)\|_2 \right). \label{prop:grad:pi:h}
\end{alignat}
The latter inequality implies \eqref{press:3:2}. In particular, \eqref{press:3:1} holds, and we have
\begin{equation*} 
\sup_{t \in [0,T]} \|\pi_h(t)\|_{1,2} \leq C_h \bigg( \sup_{t \in [0,T]} \left\{ \|\vv(t)\|_2^2 + \kappa \|\nabla \vv(t)\|_2^2 \right\} + \|\vv_0\|_2^2 + \kappa \|\nabla \vv_0\|_2^2 \bigg),
\end{equation*}
whenever \eqref{Hyp:u0} holds.

Then, taking
\begin{equation*}
  \vv(t) := \vv_1(t) + \vv_2(t) + \vv_h(t),
\end{equation*}
we obtain from \eqref{PD6},
\begin{align*}
-(\I - \kappa \Delta) \big(\vv(t) - \vv(0)\big) + \divv\big(\widetilde{\bcG}_1(t) - \widetilde{\bcG}_2(t)\big) = \nabla \pi(t).
\end{align*}
Combining this with the three previous Stokes problems, we obtain another Stokes problem, now with zero right-hand side:
\begin{equation*}
	\left\{\begin{aligned}
   -\Delta \vv(t) + \nabla\big(\widetilde{\pi}_1(t) + \widetilde{\pi}_2(t) + \pi_h(t) - \pi(t)\big) &= 0, \qquad &&\text{in } G, \\
   \div\vv(t) &= 0, \qquad &&\text{in } G, \\
   \vv(t) &= 0, \qquad &&\text{on } \partial G.
\end{aligned}
\right.
\end{equation*}
It is well-known that for this Stokes problem with zero forcing and homogeneous Dirichlet boundary conditions, one has
\begin{equation*}
\vv(t) = 0, \text{ and } \pi(t) = \widetilde{\pi}_1(t) + \widetilde{\pi}_2(t) + \pi_h(t).
\end{equation*}

For $t \in (0,T)$, we insert $\bpsi(\cdot,t) = \partial_t \bfi(\cdot,t)$, with $\supp(\bfi) \Subset G \times [0,T)$, in \eqref{PD6}, and integrate the resulting equation over $(0,T)$ to obtain
\begin{equation*} 
\begin{split}
& \!\int_0^T\!\!\!\int_G (\I - \kappa \Delta) \big(\vv(t) - \vv(0)\big) \cdot \partial_t \bfi \, \mathrm{d}x \, \mathrm{d}t  + \!\int_0^T\!\!\!\int_G \big(\widetilde{\bcG}_1(t) + \widetilde{\bcG}_2(t)\big) : \nabla \partial_t \bfi \, \mathrm{d}x \, \mathrm{d}t
\\
&=
 \!\int_0^T\!\!\!\int_G \big(\widetilde{\pi}_1(t) + \widetilde{\pi}_2(t) + \pi_h(t)\big) \div\partial_t \bfi \, \mathrm{d}x \, \mathrm{d}t.
\end{split}
\end{equation*}
Finally, integrating this equation by parts in time and invoking \eqref{PD3}, \eqref{BcG:i}, \eqref{prop:pi:i}, \eqref{der:pi}, and
\eqref{prop:pi:h}, we obtain the decomposition identity \eqref{decomp:press}.
}
\end{proof}

\section{Proof of Theorem~\ref{thm:exist:2}}\label{Sect:ex:small}

{
We begin by regularizing problem \eqref{1.1}, a necessary step to ensure that Theorem~\ref{thm:exist:1} can be applied to the resulting problem presented in \eqref{1.1:AS1} below.
We choose a Lebesgue exponent $\beta$ such that
\begin{equation}\label{def:beta}
\beta \geq \frac{d}{2}\qquad \mbox{and}\qquad d>3.
\end{equation}
For technical convenience in the arguments that follow (see~\eqref{def:r0} below), we also impose an upper bound on $\beta$.
As a result, the admissible range of $\beta$ in \eqref{def:beta} is restricted to
\begin{equation}\label{def:beta:d}
\frac{d}{2} \leq \beta \leq d\qquad \mbox{and}\qquad d>3.
\end{equation}
}
{

\subsection{Regularized problem}

For $\beta$ in the conditions of \eqref{def:beta:d} and for each $n\in\N$, we consider the following regularized problem in the spirit of~\cite{ZhikovPastukhova2011},
{\begin{equation}\label{1.1:AS1}
	\hspace{-0.485cm}\left\{	
	\begin{aligned}
		{\partial_t(\vv_n -\kappa\Delta \vv_n)+\divv\!(\vv_n\otimes \vv_n)- \nu\,\divv\!(\bfA(\vv_n))} - \frac{1}{n}\,\divv\!(\bfB(\vv_n))  & = \f -\nabla\pi_n,\quad \text{in}\ Q_T ,  \\
		\div\vv_n &=0,\quad \text{in}\ Q_T,  \\
		\vv_n &=\vv_{0},\quad \text{in}\ \1\times\{0\},  \\
		\vv_n &=\boldsymbol{0},\quad \text{on}\ \Gamma_T,
		\end{aligned}
\right.
\end{equation}
depending on the forcing term $\f$, and on the initial velocity $\vv_0$.}
The operator $\bfA$ is defined in \eqref{op:A(u)}, and the operator $\bfB$ is defined by
\begin{equation*}
\bfB(\vv):=|\bfD(\vv)|^{\beta-2}\bfD(\vv),
\end{equation*}
for $\beta$ given in \eqref{def:beta:d}.
Note that, under assumption~\eqref{def:beta}, $\beta \geq p$ and therefore
$\L^\beta(0,T;\bfV_\beta) \subset \L^p(0,T;\bfV_p)$, which in turn implies that any weak solution to the problem~\eqref{1.1:AS1} that is in $\L^\beta(0,T;\bfV_\beta)$ is also in $\L^p(0,T;\bfV_p)$.
}

\begin{definition}\label{def.1:ap}
Let the conditions of Definition~\ref{def.1} be satisfied and assume \eqref{def:beta} holds.
We say that $\vv_n$ is a weak solution to the problem \eqref{1.1:AS1}, if in addition to (2)-(3) of Definition~\ref{def.1}:
\begin{enumerate}
  \item[(1')] $\vv_n\in\L^\infty(0,T;\bfV)\cap \L^\beta(0,T;\bfV_\beta)$;
  \item[(4')] The following identity,
\begin{equation}\label{3p3:ap}
\begin{split}
&\hspace{-2mm}
-\!\int_{0}^{T}\!\!\!\!\int_{\Omega}\vv_n\cdot\partial_t\bfi \dxt - \kappa\!\int_{0}^{T}\!\!\!\!\int_{\Omega}\nabla\vv_n:\nabla\partial_t\bfi\dxt
+
\nu\!\int_{0}^{T}\!\!\!\!\int_{\Omega}|\bfD(\vv_n)|^{p-2}\bfD(\vv_n):\bfD(\bfi)\dxt
\\
&\quad   + \frac{1}{n}\!\int_{0}^{T}\!\!\!\!\int_{\Omega}|\bfD(\vv_n)|^{\beta-2}\bfD(\vv_n):\bfD(\bfi)\dxt\\
&=
\int_\1\vv_0\cdot\bfi(0)\dx + \kappa\int_\1\nabla\vv_0:\nabla\bfi(0)\dx + \!\int_{0}^{T}\!\!\!\!\int_{\Omega}\f\cdot\bfi\dxt + \!\int_{0}^{T}\!\!\!\!\int_{\Omega} \vv_n\otimes\vv_n:\mathds{\nabla}\bfi\dxt,
\end{split}
\end{equation}
holds for all $\bfi\in \bfC^\infty_{\div}(Q_T)$ with $\supp(\bfi) \Subset \Omega\times [0,T)$.
\end{enumerate}
\end{definition}

{
Since the exponent $\beta$ satisfies \eqref{def:beta}, it also meets condition \eqref{cond:2:thm1}.
Therefore, the existence of a weak solution $\vv_n$ to the problem \eqref{1.1:AS1}, in the sense of Definition~\ref{def.1:ap},
can be directly inferred from Theorem~\ref{thm:exist:1}}.
In this case only an adaptation of assumption \eqref{Hyp:u0} is required,
\begin{equation*}
\vv_0\in\bfV\cap\bfV_\beta.
\end{equation*}
Note that assumption \eqref{Hyp:f} is still usable in the present case, because, as $\beta\geq p$, we have
$\L^{p'}(0,T;\L^{p'}(\1)^d)\subset\L^{\beta'}(0,T;\L^{\beta'}(\1)^d)$.

{ In view of \eqref{3p3:ap} and a density argument, equation \eqref{3p3:ap} may be tested with
	$\bpsi \in \L^\infty(0,T;\mathbf{V}) \break \cap \L^\beta(0,T;\mathbf{V}_\beta)$, for any $\beta \ge p$, provided that all resulting terms are well defined.
Indeed, there exists a sequence of test functions $\{\bfi_m\}_{m\in\mathbb{N}} \subset \mathbf{C}^\infty(Q_T)$, with $\operatorname{supp}(\bfi_m) \Subset \Omega \times [0,T)$, such that
	$\bfi_m \to \bpsi, \  \text{ as } \ m \to \infty$, in $\L^\infty(0,T;\mathbf{V}) \cap \L^\beta(0,T;\mathbf{V}_\beta)$.
}

\subsection{A priori estimates}
{
As we have seen in the proof Theorem~\ref{thm:exist:1} (see Section~\ref{Sect:ex:thm:1}), and as $\beta$ satisfies \eqref{cond:2:thm1}, the space of test functions coincides with the space where the solutions to the problem \eqref{1.1:AS1} are constructed.
In particular, the counterparts of \eqref{ae1:un:p} and \eqref{ae2:un':p} are also valid here in the same conditions of Lemma~\ref{thrmUE:p}, }
\begin{equation}\label{ae1:un}
\begin{split}
&
\sup_{t\in[0,T]}\left\{\|\vv_n(t)\|^2_{2} + \kappa\|\nabla\vv_n(t)\|_{2}^2\right\} + \int_0^T\left(\nu\|\nabla\vv_n(t)\|^p_{p}+
\frac{1}{n}\|\nabla\vv_n(t)\|^\beta_{\beta}\right)\dt
 \\
&\leq C_1\left(\int_0^T\|\f(t)\|^{p'}_{p'}\dt + \|\vv_0\|^2_{2} + \|\nabla\vv_0\|_{2}^2 \right),
\end{split}
\end{equation}
and
\begin{equation} \label{ae2:un'}
\begin{split}
&
 \sup_{t\in[0,T]}\left\{\|\nabla\vv_n(t)\|^p_{p}+\frac{1}{n}\|\nabla\vv_n(t)\|^\beta_{\beta}\right\} +
\int_0^T\left(\|{\partial_t\vv_{n}}(t)\|_{2}^2+\kappa\|\nabla{\partial_t\vv_{n}}(t)\|^2_{2}\right)\dt
\\
& \leq C_2\left(\int_0^T\|\f(t)\|^{p'}_{p'}\dt + \|\nabla\vv_0\|^p_{p} + \frac{1}{n}\|\nabla\vv_0\|^\beta_{\beta} \right).
\end{split}
\end{equation}

Given their importance to the remainder of this work, the following two estimates are presented as separate claims.

\begin{claim}\label{claim:3}
For any $p\geq1$, there exists a positive constant $C$ (independent of $n$) such that,
\begin{equation}\label{est:E:un2}
	  \int_{0}^{T}\|\vv_n(t)\|^{\zeta}_{\zeta}\d t\leq C, \quad \text{ for }\quad  \zeta :=\left\{
	  \begin{aligned}
	  	& \frac{pd}{d-2}, \text{ for } d\neq  2,\\
	  	&[1,\infty), \text{ for } d=2.
	  \end{aligned}
	  \right.
\end{equation}
\end{claim}
\begin{proof}[Proof of Claim~\ref{claim:3}]
Assume that $d\not=2$ (for $d=2$ is easier).

For $1\leq p\leq 2$, we can use Sobolev's and H\"older's inequalities, together with \eqref{ae1:un}, so that
\begin{equation*}
\begin{split}
\int_{0}^{T}\|\vv_n(t)\|^{\frac{pd}{d-2}}_{\frac{pd}{d-2}}\d t & \leq
C\int_{0}^{T}\|\nabla\vv_n(t)\|_2^{\frac{pd}{d-2}}\d t \\
&
\leq C(T)\bigg(\sup_{t\in[0,T]}\left\{\|\vv_n(t)\|^{2}_2+\kappa\|\nabla\vv_n(t)\|^{2}_2\right\}\bigg)^{\frac{pd}{2(d-2)}}\leq C.
\end{split}
\end{equation*}
If $p>2$, we can use the embedding \begin{align*}
	\L^\infty(0,T;\W_0^{1,2}(\1)^d)\cap\L^p(0,T;\W_0^{1,p}(\1)^d)\hookrightarrow \L^{\frac{pd}{d-2}}(0,T;\L^{\frac{pd}{d-2}}(\1)^d),
\end{align*}
together with Sobolev's, H\"older's and Korn's inequalities,  and with \eqref{ae1:un}, to show that
\begin{equation*}
\begin{split}
\int_{0}^{T}\|\vv_n(t)\|^{\frac{pd}{d-2}}_{\frac{pd}{d-2}}\d t & \leq C\int_{0}^{T}\|\nabla\vv_n(t)\|_2^{\frac{2p}{d-2}}\|\nabla\vv_n(t)\|_p^{p}\d t,\quad (p>2)  \\
& \leq C\bigg(\sup_{t\in[0,T]}\left\{\|\vv_n(t)\|^{2}_2+\kappa\|\nabla\vv_n(t)\|^{2}_2\right\}\bigg)^{\frac{2p}{2(d-2)}} \int_{0}^{T}\|\nabla\vv_n(t)\|_p^{p}\d t\leq C.
\end{split}
\end{equation*}
Hence, we obtain the required estimate \eqref{est:E:un2}.
\end{proof}

{\begin{claim}\label{claim:4}
For
\begin{align}\label{est:E:un:div2}
1\leq r_0\leq \frac{d}{d-1},
\end{align}
there exist a positive constant $C$ (independent of $n$) such that
 \begin{align}\label{est:E:un:div}
	\int_{0}^{T}\|\vv_n(t)\otimes\vv_n(t)\|^{r_0}_{r_0}\d t+\int_{0}^{T}\|\divv\!(\vv_n(t)\otimes\vv_n(t))\|^{r_0}_{r_0}\dt\leq C .
\end{align}
\end{claim}
}

\begin{proof}[Proof of Claim~\ref{claim:4}]
Assume that $d\not=2$. By H\"older's inequality and \eqref{ae1:un}, one immediately has
\begin{equation*}
\int_{0}^{T}\|\vv_n(t)\otimes\vv_n(t)\|^{r_0}_{r_0}\d t\leq \int_{0}^{T}\|\vv_n(t)\|^{2r_0}_{2r_0}\d t
\leq {  C(T)\bigg(\sup_{t\in[0,T]}\|\nabla\vv_n(t)\|^{2}_{2}\bigg)^{r_0}}\leq C,
\end{equation*}
for
\begin{equation*}
1<r_0\leq\frac{d}{d-2}.
\end{equation*}

In turn, using H\"older's and Sobolev's inequalities, together with \eqref{ae1:un}, one has
\begin{equation*}
\begin{split}
 \int_{0}^{T}\|\divv\!(\vv_n(t)\otimes\vv_n(t))\|^{r_0}_{r_0}\d t & \leq
\int_{0}^{T}\left\|\big|\vv_n(t)\big|\,\big|\nabla\vv_n(t)\big|\right\|^{r_0}_{r_0}\d t \\
& \leq \int_{0}^{T}\|\vv_n(t)\|_{\frac{2r_0}{2-r_0}}^{r_0}\|\nabla\vv_n(t)\|_{2}^{r_0}\d t\\&
\leq C \int_0^T \|\nabla \vv_n(t)\|_{2}^{2r_0}\d t, \quad \text{ for } \quad 1\leq r_0\leq \frac{d}{d-1}
\\& \leq   C(T)\sup_{t\in[0,T]}\|\nabla\vv_n(t)\|^{2}_{2r_0}
 \leq C.
\end{split}
\end{equation*}
Combining the above estimates, we prove \eqref{est:E:un:div} for $r_0\in \Big[1,\frac{d}{d-1}\Big]$.

The case $d=2$ is easier.
The first part follows directly from \eqref{est:E:un2}, while the second part is a straightforward consequence of Sobolev’s inequality.
\end{proof}

Using \eqref{ae1:un} and \eqref{ae2:un'}, together with \eqref{est:E:un:div}, we can apply the Banach-Alaoglu theorem, to pass to the limit along subsequences, still labeled by the same subscript, as follows,
		\begin{alignat}{2}
			\label{5.42:1}
			\vv_n&\xrightharpoonup[n\to\infty]{\ast} \vv,\quad && \text{in}\quad \L^\infty(0,T;\bfV),  \\
\label{5.42:21}
			\vv_n&\xrightharpoonup[n\to\infty]{\ast} \vv,\quad && \text{in}\quad \L^\infty(0,T;\bfV_p),  \\
			\label{5.42:2}
			\vv_n&\xrightharpoonup[n\to\infty]{} \vv,\quad  &&\text{in}\quad \L^{p}(0,T;\bfV_p),  \\
			\label{5.46}
			\bfA(\vv_n)&\xrightharpoonup[n\to\infty]{} \overline{\bmcA},\quad && \text{in}\quad \L^{p'}(0,T;\L^{p'}(\1)^{d\times d}),
			 \\
			\label{5.47}
			\bfA(\vv_n)&\xrightharpoonup[n\to\infty]{} \overline{\bmcA},\quad  &&\text{in}\quad \L^{p'}(0,T;\W^{-1,p'}(\1)^{d\times d}), \\
\label{5.42:8}
			{\partial_t\vv_{n}}&\xrightharpoonup[n\to\infty]{} {\partial_t\vv},\quad  &&\text{in}\quad \L^{2}(0,T;\bfV),
		\end{alignat}
From \eqref{5.42:1}, \eqref{5.42:2} and \eqref{5.42:8}, it follow assertions (1)-(2) of Definition~\ref{def.1}.
On the other hand, using Korn's inequality and \eqref{ae1:un}, one has
\begin{equation*}
\int_{0}^{T}\left\|\frac{1}{n}|\bfD(\vv_n(t))|^{\beta-2}\bfD(\vv_n(t))\right\|_{\beta'}^{\beta'}\dt\leq
\frac{1}{n^{\frac{1}{\beta-1}}}\frac{1}{n}\int_{0}^{T}\|\nabla\vv_n(t)\|^\beta_{\beta}\dt \leq \frac{C}{n^{\frac{1}{\beta-1}}},
\end{equation*}
which proves that
\begin{equation}\label{5.044}
		\frac{1}{n}|\bfD(\vv_n)|^{\beta-2}\bfD(\vv_n) \xrightarrow[n\to\infty]{} 0,\quad \text{in} \quad \L^{\beta'}(0,T;\L^{\beta'}(\1)^{d\times d}).
\end{equation}

\subsection{Pressure recovery and decomposition}\label{Subs:Press:Dec}
{
In the particular case of $p$ satisfying \eqref{3.7},} the reintroduction and decomposition of the pressure in the weak formulation constitute a structural necessity for establishing an existence result.
This procedure enables the precise separation of the inertial, viscous, and convective contributions, as stated in Theorem~\ref{thrm41}.
It further compensates for the lack of regularity of the velocity field, by providing the analytical framework required to control nonlinear terms under weaker integrability.
In addition, the decomposition plays a key role in compactness and convergence arguments, allowing one to isolate suitable regular pressure components from more singular parts, thereby ensuring strong convergence of the nonlinear expressions in the weak formulation.
Without this step, the interplay between the viscous nonlinearity and the limited compactness of the velocity field would obstruct the proof of existence.
For that purpose, let us fix
\begin{align}
	\bcG_{1,n}&: =\nu  \bfA(\vv_n)=\nu  |\bfD(\vv_n)|^{p-2}\bfD(\vv_n)\in \L^{p'}(0,T;\L^{p'}(\1)^{d\times d}), \label{G:1n}\\
	\bcG_{2,n}&:= -\nabla(-\Delta)^{-1}\f+\frac{1}{n}\bfB(\vv_n)+\vv_n\otimes\vv_n\in \L^{r_0}(0,T;\L^{r_0}(\1)^{d\times d}),
\label{G:2n}
\end{align}
for
{
\begin{equation}\label{def:r0}
r_0:=\frac{d}{d-1}.
\end{equation}
We emphasize that, in \eqref{G:2n} and \eqref{def:r0}, we have employed the continuity of
$\nabla \Delta^{-1}$ as an operator from $\L^{r_0}(\1)^d$ into
$\W^{1,r_0}(\1)^{d \times d}$, combined with assumption \eqref{Hyp:f} and the estimates
\eqref{est:E:un:div} and \eqref{est:E:un:div2}.
This yields \eqref{def:r0} with
$r_0=\min\left\{ p',\, \beta',\, \frac{d}{d-1} \right\}$.
But, due to assumption \eqref{def:beta:d}, it follows that $r_0 = \frac{d}{d-1}$.
}

{Let us consider fixed but arbitrary open bounded sets
\begin{equation*}
G \Subset G' \Subset \Omega, \quad  \text{ and } \quad \partial G' \in C^2.
\end{equation*}
}
By Theorem \ref{thrm41}, {there exist functions}
\begin{align*}
	\pi_{1,n}& \in \L^{p'}(0,T;\L^{p'}(G')),\\
	\pi_{2,n} &\in \L^{r_0}(0,T;\W^{1,r_0}(G')),\\
	\pi_{h,n} &\in \C_{\mathrm{w}}([0,T];\W^{1,2}(G')),
\end{align*}
with $\Delta \pi_{h,n}=0$ and $\pi_{h,n}(0)=0$ {(cf.~\eqref{press:3:1}).
Moreover, in view of \eqref{press:3:2}--\eqref{press:3:4}, \eqref{ae1:un} and \eqref{G:1n}--\eqref{G:2n},}
\begin{align}
\label{5.53}
 \int_0^T\|\pi_{1,n}(t)\|_{p'}^{p'}\dt &\leq C_1, \\
\label{5.54}
  \int_0^T\left(\|\pi_{2,n}(t)\|_{r_0}^{r_0}+\|\nabla\pi_{2,n}(t)\|_{r_0}^{r_0}\right)\dt &\leq C_2 , \\
\nonumber
 \sup_{t\in[0,T]}\|\pi_{h,n}(t)\|_{1,2}& \leq C_3,
\end{align}
for some positive constants $C_1$, $C_2$ and $C_3$ that do not depend on $n$.
{In particular, using \eqref{3p3:ap}, we deduce
\begin{equation}\label{5.49}
\begin{split}
&
-\!\int_{0}^{T}\!\!\!\!\int_{G'}\big(\vv_n{\partial_t\bfi} +\kappa\nabla\vv_n:\nabla{\partial_t\bfi}\big)\dxt
\\
& \quad =-
\!\int_{0}^{T}\!\!\!\!\int_{G'}(\bcG_{1,n}-\pi_{1,n}\bfI):\nabla \bfi  \dxt
-
\!\int_{0}^{T}\!\!\!\!\int_{G'} (\bcG_{2,n}-\pi_{2,n}\bfI ): \nabla\bfi \dxt
\\
&\qquad
 + \int_G'\big(\vv_{0}\cdot\bfi(0)+\kappa\nabla\vv_{0}:\nabla\bfi(0)\big)\dx +\!\int_{0}^{T}\!\!\!\!\int_{G'}\nabla\pi_{h,n}\cdot{\partial_t\bfi}\dxt,
\end{split}
\end{equation}
for all $\bfi \in \bfC^\infty(G'_T)$, where $G'_T:=G'\times (0,T)$, with $\supp(\bfi)\Subset G'\times[0,T)$, {and for $\bcG_{1,n}$ and $\bcG_{2,n}$ defined in \eqref{G:1n} and \eqref{G:2n}, respectively.}

On the other hand, {by \eqref{press:3:1}
\begin{equation}\label{press:3:h}
-\Delta \pi_{h,n}=0,\quad \mbox{in}\quad G',\quad \text{ and } \quad  \pi_{h,n}(0)=0,
\end{equation}
and so $\pi_{h,n}(t)$ is harmonic in $G'$ for any $t\in(0,T)$.}
By the local regularity theory (see
e.g.~\cite[Theorem~8.24]{GT:1998}), {and by using \eqref{press:3:2}, \eqref{prop:pi:h}--\eqref{prop:grad:pi:h}, Poincar\'e's inequality and \eqref{5.3:p}, it follows} that for every $G\Subset G'$, for all $t\in(0,T)$ and all $r\in[1,\infty]$,
\begin{equation*}
\|\pi_{h,n}(t)\|_{2,r;\,G}\leq C_1\|\pi_{h,n}(t)\|_{2;\,G'}\leq C_2\big(\|\vv_n(t)-\vv_{0,n}\|_2+\kappa\|\nabla\vv_n(t)-\nabla\vv_{0,n}\|_2\big),
\end{equation*}
for some positive constants $C_1=C_1(G,G')$ and $C_2$.
Hence, by using assumption \eqref{Hyp:u0}, {together with \eqref{uniform:cont:p}} and \eqref{5.42:1}, we have
\begin{equation}\label{5.57}
\|\pi_{h,n}\|_{\L^\infty(0,T;\W^{2,r}_\loc(G'))}\leq C,\quad \text{ for all }\  r\geq 1.
\end{equation}

Then, by using \eqref{5.53}--\eqref{5.54} and \eqref{5.57}, the Banach-Alaoglu theorem gives the existence of functions $\pi_h$, $\pi_1$ and $\pi_2$ such that,
for some subsequences, still labeled by the same subscript,
\begin{alignat}{2}
	\label{5.58}
	\pi_{h,n} &\xrightharpoonup[n\to\infty]{}  \pi_h, \text{ in }\L^r(0,T;\W_\loc^{2,r}(G')),\quad \text{ for all }\  r\geq 1,\\
	\label{5.59}
	\pi_{1,n} &\xrightharpoonup[n\to\infty]{}\pi_1, \text{ in } \L^{p'}(0,T;\L^{p'}(G')),\\
	\label{5.60}
	\pi_{2,n}&\xrightharpoonup[n\to\infty]{}   \pi_2,   \text{ in } \L^{r_0}(0,T;\W^{1,r_0}(G')),
\end{alignat}
{for $r_0$ defined in \eqref{def:r0}.}
{Since the Laplacian $\Delta: \W^{2,r}_{\rm loc}(G') \longrightarrow \L^r_{\rm loc}(G')$ is linear and continuous,
we can pass \eqref{press:3:h} to the weak limit in $\L^r_{\rm loc}(G')$ for almost every $t \in [0,T]$, so that
\begin{equation}\label{press:3:h:wl}
-\Delta \pi_{h}=0,\quad \mbox{in}\quad G', \quad \text{ and } \quad \pi_{h}(0)=0.
\end{equation}}

\subsection{Compactness}\label{CP}
{Here, we first observe that by using \eqref{3p3:ap}, together with \eqref{Hyp:f}, \eqref{def:beta:d}, \eqref{ae1:un}--\eqref{ae2:un'} and \eqref{est:E:un2}--\eqref{est:E:un:div}, we can show that
\begin{align*}
	\big\|\Delta\vv_n\big\|_{\W^{1,r_0}(0,T;\W^{-1,r_0}(\Omega)^d)} \leq C,
\end{align*}
for $r_0$ given by \eqref{def:r0}.
As a consequence,
\begin{align}\label{est:dt:vn:1,r0}
	\big\|\vv_n\big\|_{\W^{1,r_0}(0,T;\W^{1,r_0}(\Omega)^d)} \leq C.
\end{align}
Such value of $r_0$ implies $r_0' = d$, and thus we have
{
\begin{equation*}
\W_0^{1,r_0}(\Omega)^d\hookrightarrow\L^{\gamma}(\Omega)^d\hookrightarrow
\W^{-1,r_0}(\Omega)^d,\quad 1\leq\gamma\leq\frac{d}{d-2},
\end{equation*}
or for any finite $\gamma\geq 1$ if $d=2$.}
Therefore, from \eqref{est:dt:vn:1,r0}, we deduce
\begin{equation}\label{est:dt:vn:-1,r0}
\big\|\vv_n\big\|_{\W^{1,r_0}(0,T;\W^{-1,r_0}(\Omega)^d)}
\leq C.
\end{equation}
Observing again the value of $r_0$, we can see the following compact and continuous embeddings hold,
\begin{equation}\label{comp:r:min}
\W_0^{1,p}(\Omega)^d\hookrightarrow\hookrightarrow \L^{r}(\Omega)^d \hookrightarrow \W^{-1,r_0}(\Omega)^d,
\qquad 1\leq\ r < \frac{dp}{d-p},\quad p<d,
\end{equation}
or for any finite $r\geq 1$ if $p=d$.
Combining this with Lemma~\ref{lem:Sim:Am}-(1),
using the parameters $s=0$, $p=r$, $s_0=1$, $p_0=r_0$, $s_1=0$, $p_1=p$,
we obtain the compact embedding
\begin{equation}\label{comp:emb:r0:r}
\W^{1,r_0}(0,T;\W^{-1,r_0}(\Omega)^d)\cap
\L^p(0,T;\W_0^{1,p}(\Omega)^d)
\hookrightarrow\hookrightarrow
\L^r(0,T;\L^r(\Omega)^d),
\end{equation}
for every $r$ satisfying \eqref{comp:r:min}.
Hence, for such $r$ and in view of \eqref{ae1:un}, \eqref{est:dt:vn:-1,r0}, and \eqref{comp:emb:r0:r}, there exists a subsequence, still labeled by the subscript $n$, such that
\begin{equation}
\label{5.12:a.e-:n}
\vv_n \xrightarrow[n\to\infty]{} \vv, \quad \text{in}\quad \L^r(0,T;\L^r(\Omega)^d).
\end{equation}
Using the Riesz--Fischer theorem, we can extract a further subsequence, still denoted by $n$, for which
\begin{equation*}
\vv_n \xrightarrow[n \to \infty]{} \vv, \quad \text{a.e. in } \Omega_T,
\end{equation*}
and consequently,
\begin{equation}\label{5.12:a.e:n}
\vv_n \otimes \vv_n \xrightarrow[n \to \infty]{} \vv \otimes \vv, \quad \text{a.e. in } \Omega_T.
\end{equation}
Then, using \eqref{est:E:un:div} together with \eqref{5.12:a.e:n}, the same classical result applied in \eqref{lim:uxu} yields
\begin{equation}\label{5.45}
{\vv_n \otimes \vv_n \xrightharpoonup[n\to\infty]{} \vv \otimes \vv, \quad  \text{in } \L^{r_0}(0,T;\L^{r_0}(\Omega)^{d\times d})},
\end{equation}
for $r_0$ satisfying \eqref{est:E:un:div2}, and hence \eqref{def:r0}.
}}

\bigskip
{
Let us work now on the harmonic pressure term $\{\pi_{h,n}\}_{n\in\N}$, and prove that
\begin{align}\label{comp:emb:h:Lr}
	\L^\infty(0,T;\L^2(G'))\cap \{\Delta u(t)=0, \text{ for a.e. } t \}\ \hookrightarrow\hookrightarrow\  {\L^r(0,T;\L^r_{\rm{loc}}(G'))}, \text{ for all }\ r>1.
\end{align}
Note that $G\Subset G'\Subset \Omega$, and observe from \eqref{press:3:h} that for each fixed $t$, $\pi_{h,n}(t)$ is harmonic in $G'$.
By standard interior estimates for harmonic functions (see e.g.~\cite[Theorem~2.10]{GT:1998}),
\begin{equation*}
\|\pi_{h,n}(t)\|_{r,G}\leq C\|\pi_{h,n}(t)\|_{2,G'},\quad \text{ for all }\ r\geq 1,
\end{equation*}
for some positive constant $C=C(G,G')$ independent of $t$.
If $\pi_{h,n}\in \L^\infty(0,T;\L^2(G'))$, the right-hand side is uniformly bounded, giving
\begin{equation*}
\pi_{h,n}\in \L^r(0,T;\L^r(G)),\quad \text{ for all }\ r\geq 1.
\end{equation*}
By the Rellich-Kondrachov theorem, after extracting a subsequence,
\begin{equation*}
\pi_{h,n}(t)\xrightarrow[n \to \infty]{}\pi_{h}(t),\quad \mbox{in}\ \ \L^r(G),\quad \text{ for all }\ r> 1,
\end{equation*}
for a.e. $t$.
The uniform bound of $\pi_{h,n}(t)$ in $\L^r(G)$ allows us to use the Lebesgue dominated convergence theorem to conclude for a further subsequence that
\begin{equation*}
\pi_{h,n}\xrightarrow[n \to \infty]{} \pi_{h}, \quad\text{in}\ \ \L^r(0,T;\L^r(G)).
\end{equation*}
Since $G\Subset G'$ was chosen arbitrary, this proves that the embedding \eqref{comp:emb:h:Lr} is compact.
}
Hence, for some subsequence still labelled by the subscript $n$, we get from \eqref{comp:emb:h:Lr}
\begin{equation}\label{5.12:a.e-:pi:n}
\pi_{h,n} \xrightarrow[n\to\infty]{} \pi_h,\quad \text{in}\quad \L^r(0,T;\L^r_{\rm{loc}}(G')),\quad \text{ for all }\ r> 1.
\end{equation}
{
Let us now fix $r>1$ and recall again the compactly contained sets $G\Subset G'\Subset \Omega$.
For each $n$ and a.e.\ $t$, the difference
\begin{equation*}
\varpi_{h,n}(t):=\pi_{h,n}(t)-\pi_h(t)
\end{equation*}
is harmonic in $G'$.
By the interior estimate for harmonic functions (see e.g.~\cite[Theorem~9.11]{GT:1998}),
\begin{equation}\label{int:W2r}
\|\varpi_{h,n}(t)\|_{2,r,G} \leq C\,\|\varpi_{h,n}(t)\|_{r,G'}, \quad \text{for a.e.\ } t\in(0,T),
\end{equation}
for some positive constant $C=C(G,G',r)$.
Raising \eqref{int:W2r} to the $r$-th power and integrate over $(0,T)$, we get
\begin{equation*}
\int_{0}^{T}\|\varpi_{h,n}(t)\|_{2,r,G}^r\dt \leq C\int_{0}^{T}\|\varpi_{h,n}(t)\|_{r,G'}^r\dt.
\end{equation*}
By \eqref{5.12:a.e-:pi:n}, the right-hand side tends to $0$ as $n\to\infty$, hence
\begin{equation*}
\pi_{h,n}\xrightarrow[n \to \infty]{} \pi_h, \quad \text{in } \L^{r}(0,T;\W^{2,r}(G)).
\end{equation*}
Since $G\Subset G'$ is arbitrary, we conclude
\begin{equation}\label{str:con:W2r}
\pi_{h,n}\xrightarrow[n \to \infty]{} \pi_h, \quad \text{in } \L^{r}(0,T;\W^{2,r}_{\rm loc}(G')).
\end{equation}

}

\subsection{Passing to the limit $n\to\infty$}\label{PL}
Using the convergence results \eqref{5.46}, \eqref{5.42:8}, \eqref{5.044} and \eqref{5.45}, together with \eqref{uniform:cont:p}, \eqref{G:1n}--\eqref{G:2n} and \eqref{5.58}--\eqref{5.60}, we can pass to the limit $n\to\infty$ in \eqref{5.49} so that the following identity holds:
\begin{equation}\label{5.49:lim}
\begin{split}
&-\!\int_{0}^{T}\!\!\!\!\int_{G'}\big(\vv\cdot {\partial_t\bfi} +\kappa\nabla\vv:\nabla{\partial_t\bfi}\big)\dxt  \\
&\quad =
-\!\int_{0}^{T}\!\!\!\!\int_{G'}(\bcG_{1}-\pi_{1}\bfI):\nabla \bfi  \dxt
-
\!\int_{0}^{T}\!\!\!\!\int_{G'} (\bcG_{2}-\pi_{2}\bfI ): \nabla\bfi \dxt \\
&\qquad
+
\int_{G'}\big(\vv_{0}\cdot\bfi(0)+\kappa\nabla\vv_{0}:\nabla\bfi(0)\big)\dx +\!\int_{0}^{T}\!\!\!\!\int_{G'}\nabla\pi_{h}\cdot{\partial_t\bfi}\dxt,
\end{split}
\end{equation}
for all $\bfi \in \bfC^\infty(G'_T)$, with $\supp(\bfi)\Subset G'\times[0,T)$,
 where
\begin{alignat*}{2}
 {\bcG_1:=\nu\overline{\bmcA}},\ \text{ and } \
 \bcG_2:=-\nabla (-\Delta)^{-1}\f+\vv\otimes \vv.
\end{alignat*}
Writing the difference between equations \eqref{5.49:lim} and \eqref{5.49}, we have
{
\begin{equation}\label{5.86}
\begin{split}
& -\!\int_{0}^{T}\!\!\!\!\int_{G'}\Big((\vv_n-\vv+\nabla\varpi_{h,n})\cdot{\partial_t\bfi} + \kappa\nabla(\vv_n-\vv):\nabla{\partial_t\bfi}\Big)\dxt  \\
& \quad = -\!\int_{0}^{T}\!\!\!\!\int_{G'}\big(\bcH_{1,n}-\varpi_{1,n}\bfI\big):\nabla\bfi  \dxt
-\!\int_{0}^{T}\!\!\!\!\int_{G'}(\bcH_{2,n}-\varpi_{2,n}\bfI):\nabla\bfi \dxt,
\end{split}
\end{equation}
for all $\bfi \in \bfC^\infty_0(G'_T)$, and where
\begin{align}	 \label{def:theta:hn1}
		\bcH_{1,n}&:=\nu\left(\bfA(\vv_n)-\overline{\bmcA}\right),  \\  \label{def:theta:hn2}
		\bcH_{2,n}&:=\vv_n\otimes \vv_n-\vv\otimes\vv + \frac{1}{n}|\bfD(\vv_n)|^{\beta-2}\bfD(\vv_n),  \\
		\varpi_{h,n}&:=\pi_{h,n}-\pi_h,\quad \varpi_{1,n}:=\pi_{1,n}-\pi_1,\quad \varpi_{2,n}:=\pi_{2,n}-\pi_2. \label{def:theta:hn}
	\end{align}
As a consequence of \eqref{5.42:2}, \eqref{5.46}, \eqref{5.044}, \eqref{5.12:a.e-:n} and \eqref{5.45}, we have
\begin{alignat}{2}
\label{5.75}
	\vv_n -\vv	&	 \xrightarrow[n\to\infty]{} \bf0,\quad \text{ in }\quad \L^r(0,T;\L^r(\1)^d),\\
	\label{5.74}
	\vv_n -\vv	&	 \xrightharpoonup[n\to\infty]{} \bf0,\quad \text{ in }\quad \L^p(0,T;\W_0^{1,p}(\1)^d),\\
	\label{5.76}
	\bcH_{1,n}&\xrightharpoonup[n\to\infty]{} \bf0,\quad \text{ in }\quad \L^{p'}(0,T;\L^{p'}(\1)^{d\times d}),\\
	\label{5.77}
	\bcH_{2,n}&\xrightharpoonup[n\to\infty]{} \bf0,\quad \text{ in }\quad \L^{r_0}(0,T;\W^{1,r_0}(\1)^{d\times d}),
\end{alignat}
for $r_0$ and $r$ satisfying \eqref{est:E:un:div2} and \eqref{comp:r:min}, respectively.
On its turn, by \eqref{5.59}--\eqref{5.60} and \eqref{str:con:W2r}, it is true that
\begin{alignat}{2}
	\label{5.79}
	\varpi_{h,n}&\xrightarrow[n\to\infty]{} 0,\quad \text{ in }\quad {\L^r(0,T;\W^{2,r}_{\rm{loc}}(G')),\quad \text{ for all }\ r>1},\\
	\label{5.80}
	\varpi_{1,n}&\xrightharpoonup[n\to\infty]{} 0,\quad \text{ in }\quad  \L^{p'}(0,T;\L^{p'}(G')),\\
	\label{5.81}
	\varpi_{2,n}&\xrightharpoonup[n\to\infty]{} 0,\quad \text{ in }\quad \L^{r_0}(0,T;\W^{1,r_0}(G')),
\end{alignat}for $r_0$ and $r$ satisfying \eqref{est:E:un:div2} and \eqref{comp:r:min}, respectively.
Moreover, by \eqref{5.57}, we have
\begin{alignat*}{2}
	\varpi_{h,n}&\in \L^\infty(0,T;\W^{2,r}_\loc(G')),\  \text{ for all }\  r\geq 1,
\end{alignat*}
uniformly in $n$.


Let us introduce the sequence
\begin{equation}\label{eq:vv:n}
\w_n:= \vv_n+\nabla \varpi_{h,n},
\end{equation}
and the double sequence
\begin{align}\label{eq:vv:n2}
	\w_{n,j}=\w_n-\w_j.
\end{align}
{Since, by definition (see \eqref{def:theta:hn}$_1$), $\varpi_{h,n} = \pi_{h,n} - \pi_h$, and
$\pi_{h,n}(0) = 0$, $\pi_h(0) = 0$ (cf. \eqref{press:3:h} and \eqref{press:3:h:wl}),
it follows that $\varpi_{h,n}(0) = 0$ for every  $n$.
Hence,
$\nabla \varpi_{h,n}(0) = 0$.
Combining this with \eqref{1.1:AS1}$_3$, we have
\begin{align}\label{eq:vv:n3}
 \w_{n,j}(0)= \00, \qquad\mbox{for $n\geq j$}.
\end{align}
}
{In view of \eqref{est:dt:vn:1,r0}, \eqref{5.75}, \eqref{5.79} and \eqref{eq:vv:n}-\eqref{eq:vv:n2},}  we have
\begin{align}\label{eq:vv:n4}
	\w_{n,j} &\xrightharpoonup[j\to\infty]{} \bf0, \quad \text{ in }\quad \L^{r_0}(0,T;\W^{1,r_0}(G')^{d}), \\\label{eq:vv:n5}
		\w_{n,j}	&	 \xrightarrow[j\to\infty]{} \bf0,\quad \text{in}\quad\ \  \L^r(0,T;\L^r_\loc (G')^d),
\end{align}
for $r_0$ and $r$ satisfying \eqref{def:r0} and \eqref{comp:r:min}, respectively.
{From \eqref{5.86}, we deduce that for $n\geq j$}
{\begin{equation}\label{5.v.nj:p}
\begin{split}
		& -\!\int_{0}^{T}\!\!\!\!\int_{G'} \Big(\w_{n,j}\cdot{\partial_t\bfi} + \kappa\nabla(\vv_n-\vv_j):\nabla{\partial_t\bfi}\Big)\dxt  \\
		&\quad =-
		\!\int_{0}^{T}\!\!\!\!\int_{G'}\big(\bcH_{1,n,j}-\varpi_{1,n,j}\bfI\big):\nabla\bfi  \dxt
		-\!\int_{0}^{T}\!\!\!\!\int_{G'}(\bcH_{2,n,j}-\varpi_{2,n,j}\bfI):\nabla\bfi \dxt,
\end{split}
\end{equation}
for all $\bfi \in \bfC_0^\infty(G'_T)$, {and where for each $i\in\{1,2\}$
\begin{alignat}{2}
& \bcH_{i,n,j}:= \bcH_{i,n}-\bcH_{i,j}, \label{H:inj} \\
& \varpi_{i,n,j}:=\varpi_{i,n}-\varpi_{i,j} \label{varpi:inj}.
\end{alignat}
}
{Using the definition of the operator $\bfI-\kappa\Delta$ (see \eqref{PD3}), together with \eqref{press:3:h}, \eqref{press:3:h:wl}, \eqref{def:theta:hn} and \eqref{eq:vv:n}--\eqref{eq:vv:n2}, and the fact that $\div\nabla^2=\nabla\Delta$, we see that
{	\begin{alignat*}{2}
		\begin{split}
			& \!\int_{0}^{T}\!\!\!\!\int_{G'}\Big(\w_{n,j}\cdot{\partial_t\bfi} + \kappa\nabla(\vv_n-\vv_j):\nabla{\partial_t\bfi}\Big)\dxt
			 =
			\!\int_{0}^{T}\!\!\!\!\int_{G'}(\bfI-\kappa\Delta)\w_{n,j}\cdot{\partial_t\bfi}  \dxt.
		\end{split}
	\end{alignat*}}
Hence \eqref{5.v.nj:p} can be written as follows, }
{\begin{equation}\label{5.v.nj}
\begin{split}
		&
		-\!\int_{0}^{T}\!\!\!\!\int_{G'}(\bfI-\kappa \Delta)\w_{n,j}\cdot{\partial_t\bfi}  \dxt \\
		&\quad =
	-	\!\int_{0}^{T}\!\!\!\!\int_{G'}\big(\bcH_{1,n,j}-\varpi_{1,n,j}\bfI\big):\nabla\bfi  \dxt
		-\!\int_{0}^{T}\!\!\!\!\int_{G'}(\bcH_{2,n,j}-\varpi_{2,n,j}\bfI):\nabla\bfi \dxt,
\end{split}
\end{equation}
for all $\bfi \in \bfC_0^\infty(G'_T)$.}

\subsection{$\mathrm{L}^\infty$--truncation}\label{Linf:T}

{To determine the space to which the time derivative $\partial_t \w_{n,j}$ belongs, we first write the identity \eqref{5.v.nj} as follows
\begin{equation}\label{5.w.nj}
-\!\int_{0}^{T}\!\!\!\int_{G'}(\bfI-\kappa \Delta)\w_{n,j}\cdot{\partial_t\bfi} \dxt = -\!\int_{0}^{T}\!\!\!\int_{G'}\bcF_{n,j}:\nabla\bfi \dxt,
\end{equation}
for all $\bfi \in \bfC_0^\infty(G'_T)$, and where
\begin{equation}\label{F:n:j}
\bcF_{n,j} := \big(\bcH_{1,n,j}-\varpi_{1,n,j}\bfI\big) + \big(\bcH_{2,n,j} - \varpi_{2,n,j}\bfI\big).
\end{equation}
By the definition of the weak time derivative, the left-hand side of \eqref{5.w.nj} can be reinterpreted in the sense of distributions on $(0,T)$.
Combining this with integration by parts in space of the right-hand side of \eqref{5.w.nj}, we obtain
\begin{equation*}
\int_{0}^{T} \prescript{}{}{\Big\langle} (\bfI-\kappa \Delta)\partial_t\w_{n,j}, \bfi \Big\rangle \mathrm{d}t=
\int_{0}^{T} \prescript{}{}{\Big\langle} \mathbf{div}\,\bcF_{n,j}, \bfi \Big\rangle \mathrm{d}t,
\end{equation*}
for all $\bfi \in \bfC^\infty_0(G'_T)$.
Hence we can write
\begin{equation}\label{duality:F:n:j}
(\bfI - \kappa \Delta)\partial_t\w_{n,j} = \mathbf{div}\,\bcF_{n,j}
\end{equation}
in the sense of distributions over $G'_T$.
By \eqref{5.76}-\eqref{5.77}, \eqref{5.80}-\eqref{5.81} and \eqref{H:inj}-\eqref{varpi:inj}, one has
\begin{equation}\label{F:n:j:Ls}
\bcF_{n,j} \in \L^s(0,T; \L^s(G')^{d\times d}),\, \, \text{ for } \,\,  s:=\min\left\{p',r_0\right\}>1.
\end{equation}
Hence,
\begin{equation*}
\divv\bcF_{n,j} \in \L^s(0,T; \W^{-1,s}(G')^{d}).
\end{equation*}
On the other hand, by the elliptic regularity theory, $(\bfI - \kappa \Delta)$ is an isomorphism from $\W^{1,s}_0(G')^d$ to $\W^{-1,s}(G')^d$.
Therefore, its inverse operator gains $2$ derivatives in spatial regularity,
\begin{equation*}
(\bfI - \kappa \Delta)^{-1} : \W^{-1,s}(G')^d \longrightarrow \W^{1,s}_0(G')^d,
\end{equation*}
and we can write
\begin{equation*}
\partial_t\w_{n,j} = (\bfI - \kappa \Delta)^{-1} \left( \mathbf{div}\,\bcF_{n,j} \right)
\end{equation*}
in the sense of distributions over $G'_T$.
Taking into account the time integrability over $(0,T)$, we conclude that
\begin{equation}\label{dwnj:dt:Ls}
\partial_t\w_{n,j} \in \L^s(0,T; \W^{1,s}_0(G')^d).
\end{equation}
As $(\bfI-\kappa\Delta)\partial_t\w_{n,j}$ and $\mathbf{div}\,\bcF_{n,j}$ belong to $\L^{s}(0,T;\W^{-1,s}(G')^d)$, both sides of the identity \eqref{duality:F:n:j} are continuous linear functionals of $\bfi$ on $\L^{s'}(0,T;\W^{1,s'}_0(G')^d)$.
Since they agree on the dense subspace $\bfC^\infty_0(Q_T)$, we infer that
\begin{equation}\label{dual:id:F:nj}
\int_{0}^{T} \prescript{}{}{\Big\langle} (\bfI-\kappa \Delta)\partial_t\w_{n,j}, \bfi \Big\rangle \mathrm{d}t=
\int_{0}^{T} \prescript{}{}{\Big\langle} \mathbf{div}\,\bcF_{n,j}, \bfi \Big\rangle \mathrm{d}t,
\end{equation}
holds for all $\bfi\in\L^{s'}(0,T;\W^{1,s'}_0(G')^d)$.

Testing \eqref{dual:id:F:nj} directly with $\w_{n,j}$ (up to a multiplicative cut-off function, cf. \eqref{cut-off:z} below), the right-hand side is well defined only if $d\leq 4$.
This restriction comes from bounding the integral term containing $\bcH_{2,n,j}-\varpi_{2,n,j}$.
In fact, using H\"older's inequality
\begin{equation}\label{RE1}
\begin{split}
&	\!\int_{0}^{T}\!\!\!\!\int_{G'} \divv\!\big(\bcH_{2,n,j}-\varpi_{2,n,j}\bfI \big)\cdot \w_{n,j}\dxt\\
&  \leq
 \int_{0}^T \big\|\divv \big(\bcH_{2,n,j}(t)-\varpi_{2,n,j}(t)\bfI\big)\big\|_{r_0}\|\w_{n,j}(t)\|_{r_0'} \d t.
\end{split}
\end{equation}
To control $\|\w_{n,j}(t)\|_{r_0'}$ and apply \eqref{5.74}, we use the Sobolev embedding $\mathbf{V} \hookrightarrow \L^{r_0'}(G')^d$, valid when
\begin{equation}\label{eq:sobolev-range}
r_0' \leq \frac{2d}{d-2} \quad\Leftrightarrow\quad r_0 \geq \frac{2d}{d+2}.
\end{equation}
Meanwhile, \eqref{est:E:un:div2}--\eqref{est:E:un:div}, \eqref{5.77} and \eqref{5.81}
only give control of $\divv \big(\bcH_{2,n,j}(t)-\varpi_{2,n,j}(t)\bfI\big)$ in $\L^{r_0}$-norm.
However, \eqref{eq:sobolev-range} holds for $r_0$ given in \eqref{def:r0} only under the dimension constraint $d \leq 4$.
To overcome this  restriction, we define the truncation at level $L>0$
\begin{equation}\label{eq:truncation-def}
T_L(\w)=
\begin{cases}
\w, & \text{if } |\w| \leq L, \\[2ex]
L\,\dfrac{\w}{|\w|}, & \text{if } |\w| > L.
\end{cases}
\end{equation}
$T_L$ is globally Lipschitz-continuous on $\mathbb R^{d}$ with Lipschitz constant $1$ (independent of $L$, $n$, $j$).
In particular,
\begin{align}
  & T_L(\00)=\00, \nonumber \\
  & |T_L(\w)| \leq L,\qquad |T_L(\w)|\leq |\w|\qquad \forall\ \w\in\mathbb{R}^d, \label{L:inf:bnd:tr} \\
  & T_L(\w_{n,j}) \xrightarrow[L\to\infty]{} \w_{n,j}\quad \text{a.e.\  in\ $Q_T$.} \nonumber
\end{align}
Moreover,
\begin{equation}\label{nabla:TL}
\nabla\big(T_L(\w)\big)=\mathbf{J}_L\nabla\w,
\end{equation}
where $\mathbf{J}_L$ is the Jacobian of $T_L$,
\begin{equation}\label{Jac:J:L}
\mathbf{J}_L(\w)
=
\begin{cases}
\bfI, & |\w|<L, \\[1ex]
\dfrac{L}{|\w|}\left(
\bfI-\dfrac{\w\otimes\w}{|\w|^2}\right),
& |\w|>L.
\end{cases}
\end{equation}
Note that $|\mathbf{J}_L(\w)|\leq 1$ and therefore
\begin{align}
  & |\nabla\big(T_L(\w_{n,j})\big)|\leq |\nabla\w_{n,j}|, \quad \text{a.e.\  in\ $Q_T$}.  \label{L:inf:bnd:gr:tr}
\end{align}

\subsection{Minty trick}\label{MiT}
It last to show that
	\begin{align}\label{claim:5}
		\overline{\bmcA}=\bfA(\vv).
	\end{align}
With that aim, we test \eqref{dual:id:F:nj} with
\begin{equation*}
\bfi_{n,j}=\zeta\,T_L(\w_{n,j}),
\end{equation*}
where $\zeta \in C_0^\infty(G')$ is a fixed cut-off function such that
\begin{equation}\label{cut-off:z}
0 \leq \zeta \leq 1\ \ \text{in}\ \ \mathbb{R}^d,\quad
\zeta\equiv 1 \ \ \text{in}\ \ G,\quad G\Subset G'\Subset \Omega.
\end{equation}
Due to \eqref{L:inf:bnd:tr} and \eqref{cut-off:z}, $\bfi_{n,j}$ is essentially bounded and belongs to $\L^{s'}(0,T;\W^{1,s'}_0(G)^d)$.
Hence $\bfi_{n,j}$ is an admissible test function for \eqref{dual:id:F:nj}, and we have
\begin{equation}\label{dual_div:T:L}
\begin{split}
& \int_{0}^{T}\big\langle(\bfI-\kappa\Delta)\partial_t\w_{n,j},\zeta\,T_L(\w_{n,j})\big\rangle\dt
=
-\!\int_{0}^{T}\!\!\!\!\int_{G'}\bcF_{n,j}: \nabla\big(\zeta\, T_L(\w_{n,j})\big)\dxt.
\end{split}
\end{equation}
Using \eqref{nabla:TL}, we expand the left-hand side as follows,
\begin{equation*}
\begin{split}
& \int_{0}^{T}\big\langle(\bfI-\kappa\Delta)\partial_t\w_{n,j},\zeta\,T_L(\w_{n,j})\big\rangle\dt \\&=  \int_{0}^{T}\!\!\!\!\int_{G'}\zeta\,\partial_t\w_{n,j}\,T_L(\w_{n,j})\dxt  +
  \kappa\int_{0}^{T}\!\!\!\!\int_{G'}\zeta\,\nabla\partial_t\w_{n,j}:\mathbf{J}_L(\w_{n,j})\nabla\w_{n,j}\dxt  \\
  &\quad +\kappa\int_{0}^{T}\!\!\!\!\int_{G'}\nabla\partial_t\w_{n,j}:\big(T_L(\w_{n,j})\otimes\nabla\zeta\big)\,\dxt.
\end{split}
\end{equation*}
Observing that $T_L$ is the gradient of the convex potential
\begin{equation*}
\phi_L(\w):=
\begin{cases}
\dfrac{|\w|^2}{2}, & |\w|\le L,\\[1ex]
L|\w|-\dfrac{L^2}{2}, & |\w|>L,
\end{cases}
\end{equation*}
\begin{equation*}
\int_{0}^{T}\!\!\!\!\int_{G'}\zeta\,\partial_t\w_{n,j}\,T_L(\w_{n,j})\dxt=
\int_{G'}\zeta\,\phi_L(\w_{n,j}(t))\dx\bigg|_{t=0}^{t=T}\,.
\end{equation*}
Introducing the function
\begin{equation*}
\eta_L(\w):=\nabla\w:\mathbf{J}_L(\w)\nabla\w,
\end{equation*}
and noting that $\mathbf{J}_L(\w)$ is symmetric,
\begin{equation*}
\kappa\int_{0}^{T}\!\!\!\!\int_{G'}\zeta\,\nabla\partial_t\w_{n,j}:\mathbf{J}_L(\w_{n,j})\nabla\w_{n,j}\dxt=
\frac{\kappa}{2}\int_{G'}\zeta\,\eta_L(\w_{n,j}(t))\dx\bigg|_{t=0}^{t=T}\,.
\end{equation*}
Note that, for both functions above, and observing that $\mathbf{J}_L(\w)$ is symmetric positive semi-definite,
\begin{equation*}
\phi_L(\w)\geq 0,\qquad \eta_L(\w)\geq 0,
\end{equation*}
pointwise in $\x$ and for every $t$. Combining this with the above identities, and using in addition \eqref{eq:vv:n3},
the left-hand side of \eqref{dual_div:T:L} comes
\begin{equation}\label{dual:n:j:lhs}
\begin{split}
& \int_{0}^{T}\big\langle(\bfI-\kappa\Delta)\partial_t\w_{n,j},\zeta\,T_L(\w_{n,j})\big\rangle\dt
\\&=
\|\zeta\,\phi_L(\w_{n,j}(T))\|_1 + \frac{\kappa}{2}\|\zeta\,\eta_L(\w_{n,j}(T))\|_1
\\
&\quad +\kappa\int_{0}^{T}\!\!\!\!\int_{G'}\nabla\partial_t\w_{n,j}:\big(T_L(\w_{n,j})\otimes\nabla\zeta\big)\,\dxt.
\end{split}
\end{equation}

For the right-hand side of \eqref{dual_div:T:L}, we use the decomposition \eqref{F:n:j}, together with \eqref{nabla:TL}, to write
\begin{equation}\label{F:n:j:rhs}
\begin{split}
&\hspace{-2mm} -\!\int_{0}^{T}\!\!\!\!\int_{G'}\bcF_{n,j}: \nabla\big(\zeta\,T_L(\w_{n,j})\big)\dxt  \\
&=
-\!\int_{0}^{T}\!\!\!\!\int_{G'}\zeta\,\bcH_{1,n,j}:\mathbf{J}_L(\w_{n,j})\nabla\w_{n,j}\dxt
+\!\int_{0}^{T}\!\!\!\!\int_{G'}\zeta\,\varpi_{1,n,j}\Tr\big(\mathbf{J}_L(\w_{n,j})\nabla\w_{n,j}\big)\dxt \\
&\quad
+
\!\int_{0}^{T}\!\!\!\!\int_{G'}\zeta\,\divv\big(\bcH_{2,n,j}-\varpi_{2,n,j}\bfI\big)\cdot T_L(\w_{n,j})\dxt
-
\!\int_{0}^{T}\!\!\!\!\int_{G'}\bcF_{n,j}:\big(T_L(\w_{n,j})\otimes\nabla\zeta\big)\dxt.
\end{split}
\end{equation}
From \eqref{def:theta:hn1}, \eqref{eq:vv:n}--\eqref{eq:vv:n2}, \eqref{H:inj} and \eqref{Jac:J:L}, we can write the first right-hand side term as follows,
\begin{equation}\label{H1:n:j:rhs}
\begin{split}
&
-\!\int_{0}^{T}\!\!\!\!\int_{G'}\zeta\,\bcH_{1,n,j}:\mathbf{J}_L(\w_{n,j})\nabla\w_{n,j}\dxt  \\
& \quad =
-
\nu\!\int\!\!\!\!\int_{G'_{L^-}}\zeta\,\big(\bfA(\vv_n)-\bfA(\vv_j)\big):\big(\bfD(\vv_n)-\bfD(\vv_j)\big)\dxt \\
&\qquad
-
\!\int\!\!\!\!\int_{G'_{L^-}}\zeta\,\bcH_{1,n,j}: \nabla^2 \varpi_{h,n,j}\dxt
-\!\int_{0}^{T}\!\!\!\!\int_{G'_{L^+}}\zeta\,\bcH_{1,n,j}:\mathbf{J}_L(\w_{n,j})\nabla\w_{n,j}\dxt,
\end{split}
\end{equation}
where
\begin{alignat}{2}
& \varpi_{h,n,j}:=\varpi_{h,n}-\varpi_{h,j} ,\label{varpi:hnj}
\end{alignat}
and
\begin{align*}
&  G'_{L^-}:=\{(\x,t)\in G'_T:|\w_{n,j}(\x,t)|\leq L\}, \\
&  G'_{L^+}:=\{(\x,t)\in G'_T:|\w_{n,j}(\x,t)|>L\}.
\end{align*}
Substituting \eqref{dual:n:j:lhs} and \eqref{F:n:j:rhs} into \eqref{dual_div:T:L}, after first using the identity \eqref{H1:n:j:rhs}, we obtain
\begin{align}\label{Minty:ij:1}\nonumber
&
J_{0,n,j}:=\nu\!\int\!\!\!\!\int_{G'_{L^-}}\zeta\,\big(\bfA(\vv_n)-\bfA(\vv_j)\big):\big(\bfD(\vv_n)-\bfD(\vv_j)\big)\dxt  \\
&=\nonumber
-\left(\|\zeta\,\phi_L(\w_{n,j}(T))\|_1 + \frac{\kappa}{2}\|\zeta\,\eta_L(\w_{n,j}(T))\|_1\right)
-\kappa\int_{0}^{T}\!\!\!\!\int_{G'}\nabla\partial_t\w_{n,j}:\big(T_L(\w_{n,j})\otimes\nabla\zeta\big)\,\dxt \\
&\quad \nonumber
-
\!\int\!\!\!\!\int_{G'_{L^-}}\zeta\,\bcH_{1,n,j}: \nabla^2 \varpi_{h,n,j}\dxt
-\!\int_{0}^{T}\!\!\!\!\int_{G'_{L^+}}\zeta\,\bcH_{1,n,j}:\nabla\big(T_L(\w_{n,j})\big)\dxt \\
&\quad \nonumber +\!\int_{0}^{T}\!\!\!\!\int_{G'}\zeta\,\varpi_{1,n,j}\Tr\Big(\nabla\big(T_L(\w_{n,j})\big)\Big)\dxt
+
\!\int_{0}^{T}\!\!\!\!\int_{G'}\zeta\,\divv\big(\bcH_{2,n,j}-\varpi_{2,n,j}\bfI\big)\cdot T_L(\w_{n,j})\dxt \\
&\quad \nonumber
-
\!\int_{0}^{T}\!\!\!\!\int_{G'}\bcF_{n,j}:\big(T_L(\w_{n,j})\otimes\nabla\zeta\big)\dxt
\\&
=: \sum_{i=1}^7	J_{i,n,j}.
\end{align}

By \eqref{cut-off:z}, $0 \leq \zeta \leq 1$ in $\mathbb{R}^d$, $\zeta\equiv 1$ in $G$ and $G\Subset G'$, and, by \eqref{2.4}--\eqref{op:A(u)},
$\big(\bfA(\vv_n)-\bfA(\vv_j)\big):\big(\bfD(\vv_n)-\bfD(\vv_j)\big)\geq 0$.
Hence,
\begin{equation}\label{est:J0:n:j}
\begin{split}
& \nu\!\int\!\!\!\!\int_{G_{L^-}}\big(\bfA(\vv_n)-\bfA(\vv_j)\big):\big(\bfD(\vv_n)-\bfD(\vv_j)\big)\dxt\leq J_{0,n,j},
\end{split}
\end{equation}
where
\begin{align*}
&  G_{L^-}:=\{(\x,t)\in G_T:|\w_{n,j}(\x,t)|\leq L\}.
\end{align*}

We can readily see that
\begin{equation}\label{est:J1:n:j}
J_{1,n,j}\leq 0.
\end{equation}

With respect to $J_{2,n,j}$, we first observe that, by \eqref{cut-off:z},
the cut-off function $\zeta$ satisfies $\zeta\equiv 1$ on $G$, and thus $\nabla\zeta\equiv 0$ on $G$.
Therefore,
\begin{equation*}
K:=\supp\big(\nabla\zeta\big)\subset \overline{G'\setminus G}\Subset G'.
\end{equation*}
By \eqref{eq:vv:n5} and since $K\Subset G'$,
\begin{equation*}
\w_{n,j}\xrightarrow[j\to\infty]{}\bf 0, \quad\text{in } \L^{r}(0,T;\L^{r}(K)^d).
\end{equation*}
Combining this with \eqref{L:inf:bnd:tr}$_2$,
\begin{equation}\label{5.79:L}
\|T_L(\w_{n,j})\|_{\L^{r}(0,T;\L^{r}(K)^d)}
\leq \|\w_{n,j}\|_{\L^{r}(0,T;\L^{r}(K)^d)}
\xrightarrow[j\to\infty]{}0, \quad (n\geq j),
\end{equation}
since $L$ is fixed.
Given that $K_T:=K\times(0,T)$ has finite measure and, by \eqref{L:inf:bnd:tr}$_1$, $|T_L(\w_{n,j})|\leq L$ pointwise, interpolation between $\L^\infty(K_T)$ and $\L^r(K_T)$ yields, for the H\"older conjugate $s'$,
\begin{equation*}
\|T_L(\w_{n,j})\|_{\L^{s'}(0,T;\L^{s'}(K)^d)}\leq
\begin{cases}
L^{1-\frac{r}{s'}}\,\|T_L(\w_{n,j})\|_{\L^{r}(0,T;\L^{r}(K)^d)}^{\frac{r}{s'}}, & s'\geq r,\\[2mm]
|K_T|^{\frac{1}{s'}-\frac{1}{r}}\,\|T_L(\w_{n,j})\|_{\L^{r}(0,T;\L^{r}(K)^d)}, & s'< r.
\end{cases}
\end{equation*}
In both cases, \eqref{5.79:L} implies
\begin{equation*}
T_L(\w_{n,j})\xrightarrow[j\to\infty]{}\bf 0, \quad\text{ in }\ \L^{s'}(0,T;\L^{s'}(K)^d).
\end{equation*}
Since $T_L(\w_{n,j})\otimes\nabla\zeta\equiv \00$ outside $K_T$,
\begin{equation*}
\|T_L(\w_{n,j})\otimes\nabla\zeta\|_{\L^{s'}(0,T;\L^{s'}(G')^{d\times d})}
\leq C\|T_L(\w_{n,j})\|_{\L^{s'}(0,T;\L^{s'}(K)^d)}
\xrightarrow[j\to\infty]{}0.
\end{equation*}
In view of this, we have, by \eqref{dwnj:dt:Ls} and H\"older's inequality,
\begin{equation}\label{est:J2:n:j}
\begin{split}
|J_{2,n,j}|\leq & \|\nabla\partial_t \w_{n,j}\|_{\L^s(0,T;\L^s(G')^{d\times d })}\|T_L(\w_{n,j})\otimes\nabla\zeta\|_{\L^{s'}(0,T;\L^{s'}(G')^{d\times d})} \\
\leq &
C\|T_L(\w_{n,j})\otimes\nabla\zeta\|_{\L^{s'}(0,T;\L^{s'}(G')^{d\times d})} \xrightarrow[j\to\infty]{}0,\qquad (n\geq j).
\end{split}
\end{equation}
The same argument applied to $J_{7,n,j}$, with \eqref{F:n:j:Ls} in place of \eqref{dwnj:dt:Ls}, yields
\begin{equation}\label{est:J7:n:j}
\begin{split}
|J_{7,n,j}|\leq & \|\bcF_{n,j}\|_{\L^s(0,T;\L^s(G')^{d\times d})}\|T_L(\w_{n,j})\otimes\nabla\zeta\|_{\L^{s'}(0,T;\L^{s'}(G')^{d\times d})} \\
\leq &
C\|T_L(\w_{n,j})\otimes\nabla\zeta\|_{\L^{s'}(0,T;\L^{s'}(G')^{d\times d})} \xrightarrow[j\to\infty]{}0,\qquad (n\geq j).
\end{split}
\end{equation}

Let us now analyze the term $J_{3,n,j}$.
Combining \eqref{5.76} and \eqref{H:inj} with the Banach-Steinhaus theorem,
\begin{equation*}
\|\bcH_{1,n,j}\|_{\L^{p'}(0,T;\L^{p'}(\1)^{d\times d})} \leq C,
\end{equation*}
uniformly with respect to both $n$ and $j$.
From \eqref{varpi:hnj} and \eqref{5.79},
\begin{equation*}
\varpi_{h,n,j} \xrightarrow[n\to\infty]{} 0,\quad \text{ in }\quad {\L^p(0,T;\W^{2,p}_{\rm{loc}}(G'))}.
\end{equation*}
As a consequence, we obtain by H\"older's inequality,
\begin{equation}\label{est:J3:n:j}
\begin{split}
\big|J_{3,n,j}\big|\leq &
\|\bcH_{1,n,j}\|_{\L^{p'}(0,T;\L^{p'}(G')^{d\times d})}\|\zeta\nabla^2\varpi_{h,n,j}\|_{\L^{p}(0,T;\L^{p}(G')^{d\times d})} \\
\leq &
C\|\nabla^2\varpi_{h,n,j}\|_{\L^{p}(0,T;\L^{p}(K)^{d\times d})}
\xrightarrow[n\to\infty]{} 0,
\end{split}
\end{equation}
for fixed $L$, and where $K:=\supp(\zeta)\Subset G'$.

For the term $J_{4,n,j}$, we note that by \eqref{L:inf:bnd:gr:tr} and \eqref{eq:vv:n}--\eqref{eq:vv:n2},
\begin{equation*}
\begin{split}
&
\int_0^T\left\|\nabla\big(T_L(\w_{n,j}(t))\big) \right\|_{\L^{p}(K)^{d\times d}}^p \dt \\&\leq
\int_0^T\left\|\nabla\w_{n,j}(t)\right\|_{\L^{p}(K)^{d\times d}}^p \dt  \\
&\leq
C\left(\int_0^T\left\|\nabla\big(\vv_{n}(t)-\vv_{j}(t)\big)\right\|_{\L^{p}(\Omega)^{d\times d}}^p +
\int_0^T\left\|\nabla^2\big(\varpi_{h,n}(t)-\varpi_{h,j}(t)\big)\right\|_{\L^{p}(K)^{d\times d}}^p\right) \dt \leq C,
\end{split}
\end{equation*}
uniformly with respect to $n$ and $j$, again for $K:=\supp(\zeta)\Subset G'$.
Here, the terms on the right-hand side are bounded by applying the Banach–Steinhaus theorem, using \eqref{5.74} for the first term and \eqref{5.79} with $r = p$ for the second.
By the Banach-Alaoglu theorem, we have, up to some subsequence,
\begin{equation*}
\nabla\big(T_L(\w_{n,j})\big) \xrightharpoonup[j\to\infty]{} \boldsymbol{\Upsilon}, \quad \text{in}\ \ \L^{p}(0,T; \L^{p}(K)^{d\times d}).
\end{equation*}
Since $T_L(\mathbf{0})=\mathbf{0}$ and $|T_L(\w )| \leq L$, the strong convergence \eqref{eq:vv:n5} and Lebesgue's dominated convergence theorem ensure that
\begin{equation*}
T_L(\w_{n,j}) \xrightarrow[j\to\infty]{} \mathbf{0}, \quad \text{in}\ \ \L^1(K_T)^d.
\end{equation*}
For any test function $\boldsymbol{\Phi} \in C_0^\infty(K_T)^{d\times d}$, integration by parts gives
\begin{equation*}
\lim_{j\to\infty} \int_0^T\!\!\int_{K} \nabla\big( T_L(\w_{n,j})\big) : \boldsymbol{\Phi} \,\dxt =
-\lim_{j\to\infty} \int_0^T\!\!\int_{K} T_L(\w_{n,j}) \cdot \divv\boldsymbol{\Phi} \,\dxt = 0.
\end{equation*}
Hence $\boldsymbol{\Upsilon} = \mathbf{0}$ a.e., and therefore
\begin{equation*}
\nabla \Big( T_L(\w_{n,j}) \Big) \xrightharpoonup[j\to\infty]{} \mathbf{0}, \quad\text{in}\ \ \L^{p}(0,T; \L^{p}(K)^{d\times d}).
\end{equation*}
On the other hand, by \eqref{5.76}, \eqref{H:inj} and Banach-Steinhaus theorem,
\begin{equation*}
\|\bcH_{1,n,j}\|_{\L^{p'}(0,T; \L^{p'}(G')^{d\times d})} \leq C,
\end{equation*}
uniformly in $n$ and $j$.
As a consequence,
\begin{equation}\label{est:J4:n:j}
\begin{split}
J_{4,n,j} \xrightarrow[j\to\infty]{} 0.
\end{split}
\end{equation}
By the same reasoning, we have
\begin{equation}\label{est:J5:n:j}
\begin{split}
J_{5,n,j} \xrightarrow[j\to\infty]{} 0,
\end{split}
\end{equation}
using, in this case, \eqref{varpi:inj}, and \eqref{5.80} instead of \eqref{5.76}.

With respect to the term $J_{6,n,j}$, we firstly observe that, due to \eqref{5.77}, \eqref{5.81} and \eqref{H:inj}--\eqref{varpi:inj}, we have
\begin{equation}\label{5.78}
\left\|\divv\big(\bcH_{2,n,j}-\varpi_{2,n,j}\bfI\big)\right\|_{\L^{r_0}(0,T;\L^{r_0}(G')^d)}\leq C,
\quad \text{for all } n\ge j,
\end{equation}
with $C$ independent of $n$ and $j$, and for $r_0$ given by \eqref{def:r0}.
Assuming the compatibility condition $r\leq r_0'$, using the $\L^r$--$\L^\infty$ interpolation inequality with $\theta=\frac{r}{r_0'}\in(0,1]$, and then \eqref{L:inf:bnd:tr}$_1$ and \eqref{5.79:L}, we have
\begin{equation} \label{5.80:L}
\begin{split}
\|\zeta\,T_L(\w_{n,j})\|_{\L^{r_0'}(0,T;\L^{r_0'}(G')^d)} \leq & \|T_L(\w_{n,j})\|_{\L^{r}(0,T;\L^{r}(K)^d)}^{\theta}\|T_L(\w_{n,j})\|_{\L^{\infty}(Q_T)}^{1-\theta} \\
\leq & L^{1-\theta}\,\|T_L(\w_{n,j})\|_{\L^{r}(0,T;\L^{r}(K)^d)}^{\theta} \xrightarrow[j\to\infty]{}0,
\end{split}
\end{equation}
for $K=\supp(\zeta)\Subset G'$.
By H\"older's inequality, \eqref{5.78} and \eqref{5.80:L},
\begin{equation}\label{est:J6:n:j}
\begin{split}
\left|J_{6,n,j}\right| & =
\left|\int_0^T\!\!\!\int_{G'} \zeta\,\divv\big(\bcH_{2,n,j}-\varpi_{2,n,j}\bfI\big)\cdot T_L(\w_{n,j})\dxt\right| \\
&
\leq \left\|\divv\big(\bcH_{2,n,j}-\varpi_{2,n,j}\bfI\big)\right\|_{\L^{r_0}(0,T;\L^{r_0}(G')^d)}\,
\|\zeta\,T_L(\w_{n,j})\|_{\L^{r_0'}(0,T;\L^{r_0'}(G')^d)} \\
&
\leq C\|T_L(\w_{n,j})\|_{\L^{r_0'}(0,T;\L^{r_0'}(K)^d)}\xrightarrow[j\to\infty]{}0.
\end{split}
\end{equation}
Note that the interpolation inequality used above required $r\leq r_0'$, where, from \eqref{def:r0}, $r_0'= d$.
Therefore the restriction $r\leq r_0'$ holds if  $r\leq d$.
We verify when $\frac{dp}{d-p}\le d$.
In our setting, $1<p<\frac{d}{2}$, cf. \eqref{3.7}$_1$, so this inequality is in fact strict.
Since $p<d$, we have $d-p>0$, and thus
$\frac{dp}{d-p}< d \Leftrightarrow p< \frac{d}{2}$.
As $r<\frac{dp}{d-p}$, cf. \eqref{comp:r:min}, we conclude that
$r < \frac{dp}{d-p} < d \leq r_0'$
and so the compatibility condition $r\leq r_0'$ holds for every $d>3$.

Using the information of \eqref{est:J0:n:j}, \eqref{est:J1:n:j}, \eqref{est:J2:n:j}--\eqref{est:J5:n:j} and \eqref{est:J6:n:j} into \eqref{Minty:ij:1}, we obtain
}}
\begin{align}\label{RS4}
	\limsup_{j\to\infty} \int\!\!\!\!\int_{G_{L^-}}\zeta\,\big(\bfA(\vv_n)-\bfA(\vv_j)\big):\big(\bfD(\vv_n)-\bfD(\vv_j)\big)\dxt\leq 0.
\end{align}
Employing the monotonicity of the nonlinear operator $\bfA(\cdot)$  (see \eqref{2.4}) and \eqref{RS4}, we find
\begin{align}\label{RS6}
		\limsup_{j\to\infty} \int\!\!\!\!\int_{G_{L^-}} \big(\bfA (\vv_n)-\bfA(\vv_j)\big):\big(\bfD(\vv_n)-\bfD(\vv_j) \big)\dxt = 0.
\end{align}
{Now, set
\begin{equation*}
\mathcal{F}_{n,j}:=\big(\bfA (\vv_n)-\bfA(\vv_j)\big):\big(\bfD(\vv_n)-\bfD(\vv_j) \big).
\end{equation*}
Combining \eqref{RS6} with Fatou's lemma, there exists a subsequence $\mathcal{F}_{n,j_k}$ such that
\begin{equation*}
\mathcal{F}_{n,j_k}\xrightarrow[k\to\infty]{}0,\qquad \text{for}\quad \lambda^{d+1}\text{-a.e. }(\x,t)\in G_{L^-}.
\end{equation*}
Using the strict monotonicity of $\bfA(\cdot)$, we conclude that
\begin{equation*}
\bfD(\vv_{j_k})\xrightarrow[k\to\infty]{} \bfD(\vv_n),\qquad \text{for}\quad \lambda^{d+1}\text{-a.e. }(\x,t)\in G_{L^-}.
\end{equation*}
Finally, for any $k,\ell$,
\begin{equation*}
|\bfD(\vv_{j_k})-\bfD(\vv_{j_\ell})|
\leq
|\bfD(\vv_{j_k})-\bfD(\vv_n)|
+
|\bfD(\vv_n)-\bfD(\vv_{j_\ell})| ,
\end{equation*}
and both terms on the right-hand side converge to zero $\lambda^{d+1}$-a.e.
Therefore, $\{\bfD(\vv_j)\}_{j\in\N}$ is a Cauchy sequence $\lambda^{d+1}$-almost everywhere.
Thus, the limit function exists $\lambda^{d+1}$-a.e.~but has to be equal to $\bfD(\vv)$ on account of \eqref{5.74}, that is,
\begin{align*}
	\bfD(\vv_n) \xrightarrow[n\to\infty]{} \bfD(\vv),\qquad \text{for}\quad \lambda^{d+1}\text{-a.e. }(\x,t)\in G_{L^-}.
\end{align*}
We address the reader to \cite[Eqn. (1.6)]{LBJMMS} or \cite{JFJMMS} for more details.

{By local identification on the truncated domain,
\begin{equation*}
\overline{\bmcA}=\bfA(\vv),\qquad \text{for}\quad \lambda^{d+1}\text{-a.e. }(\x,t)\in G_{L^-}.
\end{equation*}
Since $\vv$ is finite $\lambda^{d+1}$-a.e. in $Q_T$, we have
\begin{equation*}
\lambda^{d+1}\left(Q_T \setminus \bigcup_{L=1}^\infty G_{L^-}\right) = 0.
\end{equation*}
Taking the countable union of exceptional null sets over all $L \in \mathbb{N}$, we conclude that
\begin{equation*}
\overline{\bmcA}=\bfA(\vv),\qquad \text{for}\quad \lambda^{d+1}\text{-a.e. }(\x,t)\in Q_T.
\end{equation*}
}

For a detailed proof of the approach used here, we address the reader to \cite[Appendix A]{Wolf:2007} or \cite[pp. 324]{Breit:2015}).
}
This result proves \eqref{claim:5} and completes the proof of Theorem~\ref{thm:exist:2}.

\section{Uniqueness}\label{Sect:unique}
{The final task in this work is to prove Theorem~\ref{UniT}, which asserts the uniqueness of weak solutions to the problem \eqref{1.1}.
}

{
\begin{proof}[Proof of Theorem~\ref{UniT}]
Let $\boldsymbol{v}_1, \boldsymbol{v}_2$ be two weak solutions  to the problem \eqref{1.1} in the sense of Definition~\ref{def.1}, with the same initial data $\boldsymbol{v}_0$ and forcing term $\boldsymbol{f}$.
Define
\begin{equation*}
\boldsymbol{w} := \boldsymbol{v}_1 - \boldsymbol{v}_2,
\end{equation*}
and note that $\div\w=0$ in $Q_T$ and $\w(0)=\bf0$ in $\Omega$.
Since $\w$ is not admissible as a test function in \eqref{3p3}, we introduce a mollification $\w^\varepsilon := \rho_\varepsilon * \w$, where the convolution is taken with respect to the time variable.
In particular, $\rho_\varepsilon(t)$ is supported in $(-\varepsilon,\varepsilon)$.
Then $\w^\varepsilon \in \bfC^\infty([0,T];\bfV_p)$, $\div\w^\varepsilon=0$, $\w^\varepsilon(\cdot,T)=0$ for $\varepsilon$ small enough, and
\begin{equation}
\label{sc:varep}
\w^\varepsilon \xrightarrow[\varepsilon\to0]{} \w,\quad \text{in}\quad \L^2(0,T;\bfV)\cap\L^p(0,T;\bfV_p).
\end{equation}
Using $\boldsymbol{w}^\varepsilon$ as a test function in \eqref{3p3} and subtracting the two weak formulations gives
\begin{align*}
& \int_0^t\!\!\int_\Omega \partial_\tau\boldsymbol{w} \cdot \boldsymbol{w}^\varepsilon \, \d \x\d\tau
+ \kappa \int_0^t\!\!\int_\Omega \nabla \partial_\tau\boldsymbol{w} : \nabla \boldsymbol{w}^\varepsilon \, \d \x\d\tau\ \\&\  \ +  \nu \int_0^t\!\!\int_\Omega \big(\bfA(\vv_1)-\bfA(\vv_2)\big):\mathbf{D}(\boldsymbol{w}^\varepsilon) \, \d \x\d\tau  \\
&= \int_0^t\!\!\int_\Omega (\boldsymbol{v}_1 \otimes \boldsymbol{v}_1 - \boldsymbol{v}_2 \otimes \boldsymbol{v}_2) : \nabla \boldsymbol{w}^\varepsilon \, \d \x\d\tau.
\end{align*}
Passing to the limit $\varepsilon \to 0$, using for this purpose \eqref{sc:varep}, this yields
\begin{equation*} 
\begin{split}
& \frac{1}{2}\|\w(t)\|_{2}^2 + \frac{\kappa}{2}\|\nabla \w(t)\|_{2}^2 + \nu \int_0^t\!\!\int_\Omega \big(\bfA(\vv_1)-\bfA(\vv_2)\big):\bfD(\w)\,\d\x\d\tau  \\
&=  \int_0^t\!\!\int_\Omega (\vv_1\otimes\vv_1-\vv_2\otimes\vv_2):\mathds{\nabla}\w\d\x\d\tau.
\end{split}
\end{equation*}
\\
Combining the identity
\begin{equation*}
\vv_1\otimes\vv_1-\vv_2\otimes\vv_2= \vv_1\otimes\w + \w\otimes\vv_2
\end{equation*}
with the  monotonicity of $\bfA(\cdot)$ (see \eqref{2.4}), we obtain
\begin{equation}\label{dif:ws:v1-v2-1}
\begin{split}
& \frac{1}{2}\|\w(t)\|_{2}^2 + \frac{\kappa}{2}\|\nabla \w(t)\|_{2}^2 \leq
\int_0^t\!\!\int_\Omega \vv_1\otimes\w:\mathds{\nabla}\w\d\x\d\tau  + \int_0^t\!\!\int_\Omega \w\otimes\vv_2:\mathds{\nabla}\w\d\x\d\tau .
\end{split}
\end{equation}
Using integration by parts over $\Omega$, the facts that $\div\vv_1=0$ in $Q_T$ and $\vv_1=0$ on $\partial\Omega$ in the sense of traces, the first right-hand side term vanishes.
For the second right-hand side term, we divided the proof into two parts, which heavily rely on the dimensions, that $2\leq d \leq 4$ or $d>4$.

\vspace{2mm}
\noindent
\textbf{Part 1.} We first consider the case $p>1$ when $2\leq d  \leq 4$.
 For the second right-hand side term appearing in \eqref{dif:ws:v1-v2-1}, we use integration by parts, H\"older's and Sobolev's inequalities, and the fact that $\boldsymbol{v}_2\in\L^\infty(0,T;\bfV)$.
After all, we obtain
\begin{equation*}
\begin{split}
\left|\int_0^t \int_\Omega \boldsymbol{w} \otimes \boldsymbol{v}_2 : \nabla \boldsymbol{w}\d\x\d\tau\right| & =
\left|-\int_0^t \int_\Omega (\boldsymbol{w}\cdot\nabla)\boldsymbol{v}_2\cdot\boldsymbol{w}\d\x\d\tau \right|\leq
 \int_0^t \|\nabla\boldsymbol{v}_2(\tau) \|_2 \|\boldsymbol{w}(\tau) \|_4^2 \d\tau \\
& \leq  C\int_0^t \|\nabla\boldsymbol{w}(\tau) \|_2^2 \d\tau, \ \  d\leq 4,
\end{split}
\end{equation*}
for some positive constant $C$.
Using this information in \eqref{dif:ws:v1-v2-1}, we get
\begin{equation}\label{dif:ws:v1-v2:a}
\|\nabla \w(t)\|_{2}^2 \leq C\int_0^t \|\nabla\boldsymbol{w}(\tau) \|_2^2 \d\tau,
\end{equation}
for a distinct positive constant $C$ and $2\leq d\leq 4$.
Applying Gr\"onwall's inequality to \eqref{dif:ws:v1-v2:a}, we obtain
\begin{equation*}
	\|\nabla \w(t)\|_2^2 \leq \|\nabla \w(0)\|_2^2 e^{Ct},\ \text{ for all }\ \ t\in[0,T].
\end{equation*}
As $\w(0)=\boldsymbol{0}$, Sobolev's inequality implies that $\vv_1=\vv_2$ a.e. in $\Omega$ and for all $t\in [0,T]$ and hence the weak solution to the system \eqref{1.1} is unique for $p>1$ and $2\leq d\leq 4$.

\vspace{2mm}
\noindent
\textbf{Part 2.} Now, we consider the case $p\geq \frac{d}{2}$ when $d> 4$.
Again, we estimate the second right-hand side term appearing
in \eqref{dif:ws:v1-v2-1} using integration by parts, the fact
that $\div\w=0$ in $Q_T$, H\"older's and Sobolev's inequalities as
\begin{align}\label{P21}\nonumber
    & \left|\int_0^t \int_\Omega \boldsymbol{w} \otimes \boldsymbol{v}_2
    : \nabla \boldsymbol{w}\,\d\x\d\tau\right|
    =
    \left|-\int_0^t \int_\Omega (\boldsymbol{w}\cdot\nabla)
    \boldsymbol{v}_2\cdot\boldsymbol{w}\,\d\x\d\tau \right|
    \\& \leq  \nonumber
    \int_0^t \|\nabla\boldsymbol{v}_2(\tau) \|_{\frac{d}{2}}
    \|\boldsymbol{w}(\tau) \|_{\frac{2d}{d-2}}^2 \, \d\tau
  \leq
    C \int_0^t \|\nabla\boldsymbol{v}_2(\tau) \|_{\frac{d}{2}}
    \|\nabla \boldsymbol{w}(\tau) \|_{2}^2 \, \d\tau
    \\& \leq 
    C \int_0^t \|\nabla\boldsymbol{v}_2(\tau) \|_{p}
    \|\nabla \boldsymbol{w}(\tau) \|_{2}^2 \, \d\tau,
\end{align}
where in the last step we used the continuous embedding
$\L^p(\Omega)\hookrightarrow \L^{\frac{d}{2}}(\Omega)$,
which holds since $p\geq \frac{d}{2}$ and $\Omega$ is bounded.
Using \eqref{P21} in \eqref{dif:ws:v1-v2-1}, we find
\begin{equation}\label{P24}
    \|\nabla \w(t)\|_{2}^2 \leq
    C\int_0^t\|\nabla \boldsymbol{v}_2(\tau)\|_{p}
    \|\nabla\boldsymbol{w}(\tau) \|_2^2 \, \d\tau.
\end{equation}
Applying Gr\"onwall's inequality to \eqref{P24}, we get
\begin{align*}
    \|\nabla \boldsymbol{w}(t)\|_2^2
    &\leq \|\nabla \boldsymbol{w}(0)\|_2^2
    \exp \bigg\{ C\int_0^t \|\nabla \boldsymbol{v}_2(\tau)\|_{p}
    \, \d\tau \bigg\}.
\end{align*}
Since $\boldsymbol{v}_2 \in \L^p(0,T;\bfV_p)$ and
$p \geq \frac{d}{2} > 2$ for $d > 4$, H\"older's inequality
in time gives
\begin{align*}
    \int_0^T \|\nabla \boldsymbol{v}_2(\tau)\|_{p} \, \d\tau
    \leq T^{1 - \frac{1}{p}}
    \left( \int_0^T \|\nabla \boldsymbol{v}_2(\tau)\|_{p}^{p}
    \, \d\tau \right)^{\!\frac{1}{p}}
    < \infty.
\end{align*}
Therefore,
\begin{align*}
    \|\nabla \boldsymbol{w}(t)\|_2^2
    \leq \|\nabla \boldsymbol{w}(0)\|_2^2 \, e^{C},
    \quad \text{for all } t \in [0,T],
\end{align*}
for some finite constant $C = C(\nu, p, d, T, \boldsymbol{v}_2) > 0$.
As $\boldsymbol{w}(0)=\boldsymbol{0}$, Sobolev's inequality implies
$\boldsymbol{v}_1=\boldsymbol{v}_2$ a.e.\ in $\Omega$ and for all
$t\in [0,T]$.
Hence the weak solution to the system \eqref{1.1}
is unique for $p\geq \frac{d}{2}$ when $d>4$.
\end{proof}

\begin{remark}
The case $1 < p < \frac{d}{2}$ with $d > 4$ remains open.
The key obstruction is that the embedding
$\L^p(\Omega) \hookrightarrow \L^{\frac{d}{2}}(\Omega)$ fails for
$p < \frac{d}{2}$, so the estimate \eqref{P21} is no longer
available. 
An interpolation argument between
$\|\nabla\boldsymbol{v}_2\|_2$ and $\|\nabla\boldsymbol{v}_2\|_p$
to bound $\|\nabla\boldsymbol{v}_2\|_{\frac{d}{2}}$ would require
$\theta \in (0,1)$, but one can verify that the resulting
exponent $\theta = \frac{p(d-4)}{d(p-2)}$ does not lie in
$(0,1)$ for $1 < p < \frac{d}{2}$ and $d > 4$.
\end{remark}

	\medskip\noindent
\textbf{Acknowledgments:} A. Kumar acknowledges partial funding by the Austrian Science Fund (FWF), Project Number: P 34681 (\href{https://www.fwf.ac.at/en/research-radar/10.55776/P34681}{10.55776/P34681}) and \href{https://www.fwf.ac.at/forschungsradar/10.55776/ESP4373225}{10.55776/ESP4373225}.
The work of H. B. de Oliveira was supported by CIDMA (https://ror.org/05pm2mw36)
under the Portuguese Foundation for Science and Technology
(FCT, https://ror.org/00snfqn58), Grant
UID/04106/2025 (\url{https://doi.org/10.54499/UID/04106/2025}), and by the Grant no. AP23486218 of the Ministry of Science and Higher Education of the Republic of Kazakhstan.
Support for M. T. Mohan's research received from the National Board of Higher Mathematics (NBHM), Department of Atomic Energy, Government of India (Project No. 02011/13/2025/NBHM(R.P)/R\&D II/1137). The first author expresses gratitude to Prof.~Dominic Breit, University of Duisburg-Essen for clarifying doubts about his article \cite{Breit:2015}, assistance that benefited this manuscript. The authors express their sincere gratitude to Prof.~Steve Shkoller, University of California for his insightful feedback, which helped to significantly improve the manuscript.

\medskip\noindent	\textbf{Declarations:}

\noindent 	\textbf{Ethical Approval:}   Not applicable.

\noindent  \textbf{Conflict of interest statement: } On behalf of all authors, the corresponding author states that there is no conflict of interest.


\noindent 	\textbf{Authors' contributions: } All authors have contributed equally.

\noindent 	\textbf{Funding:} Austrian Science Fund (FWF), Project Number: P 34681 (\href{https://www.fwf.ac.at/en/research-radar/10.55776/P34681}{10.55776/P34681})  and \href{https://www.fwf.ac.at/forschungsradar/10.55776/ESP4373225}{10.55776/ESP4373225} (A. Kumar).\\
Portuguese Foundation for Science and Technology: Grant UID/04106/2025, \url{https://doi.org/10.54499/UID/04106/2025}, and Ministry of Science and Higher Education of the Republic of Kazakhstan: Grant AP23486218 (H. B. de Oliveira). \\
DST-SERB, India, MTR/2021/000066 (M. T. Mohan).

\noindent 	\textbf{Availability of data and materials: } Not applicable.


\bibliographystyle{plain}
\bibliography{Power_law_NSV}

\end{document}